\documentclass[12pt]{amsart}
\usepackage{amsmath}
\usepackage{amsthm}
\usepackage{amssymb}
\usepackage{enumerate}
\usepackage{xcolor}
\usepackage{comment}
\newtheorem{theorem}{Theorem}[section]
\newtheorem{lemma}[theorem]{Lemma}
\newtheorem{proposition}[theorem]{Proposition}
\newtheorem{corollary}[theorem]{Corollary}

\newtheorem{remar}[theorem]{Remark}
\theoremstyle{definition}

\newtheorem{prob}[theorem]{Open Problem}

\newenvironment{remark}{\begin{remar}\rm}{\gzaun\end{remar}}
\newcommand{\gzaun}{\unskip\nobreak\hfil\penalty50%
\hskip1em\hbox{}\nobreak\hfil%
$\#$\parfillskip=0pt\finalhyphendemerits=0}   

\newcommand{\bfind}[1]{\index{#1}{\bf #1}}
\newcommand{\n}{\par\noindent}
\newcommand{\sn}{\par\smallskip\noindent}
\newcommand{\mn}{\par\medskip\noindent}
\newcommand{\bn}{\par\bigskip\noindent}
\newcommand{\pars}{\par\smallskip}
\newcommand{\parm}{\par\medskip}

\newcommand{\cal}{\mathcal}
\newcommand{\isom}{\simeq}
\newcommand{\ovl}[1]{\overline{#1}}

\newcommand{\sep}{^{\rm sep}}
\newcommand{\chara}{\mbox{\rm char}\,}

\newcommand{\Gal}{\mbox{\rm Gal}\,}

\newcommand{\erf}{\mbox{\rm crf}\,}
\newcommand{\dist}{\mbox{\rm dist}\,}

\newcommand{\cO}{\mathcal{O}}
\newcommand{\cM}{\mathcal{M}}
\newcommand{\cE}{\mathcal{E}}
\newcommand{\R}{\mathbb R}
\newcommand{\Q}{\mathbb Q}
\newcommand{\N}{\mathbb N}
\newcommand{\Z}{\mathbb Z}
\newcommand{\F}{\mathbb F}
\newcommand{\Qp}{{\mathbb Q}_p}

\begin{document}
\title[Deeply ramified fields and defect extensions]{The valuation
theory of deeply ramified fields and its connection with defect extensions}
%
\author{Franz-Viktor Kuhlmann}
\address{Institute of Mathematics, University of Szczecin,	
ul. Wielkopolska 15, 	  	  	
70-451 Szcze\-cin, Poland}
\email{fvk@usz.edu.pl}

\author{Anna Rzepka}
\address{Institute of Mathematics, University of Silesia in Katowice, Bankowa 14,
40-007 Katowice, Poland}
\email{anna.rzepka@us.edu.pl}
\date{11.\ 11.\ 2022}
\thanks{The authors would like to thank the referees for many very useful corrections and 
suggestions that helped to improve the paper significantly. We thank Michael Temkin for
drawing our attention to the notion of deeply ramified extensions, which were introduced 
by John Coates and Ralph Greenberg in \cite{CG}.\\
The first author was partially supported by Opus grant 2017/25/B/ST1/01815 
from the National Science Centre of Poland.}

\begin{abstract}\noindent
We study in detail the valuation theory of deeply ramified fields and introduce and
investigate several other related classes of valued fields. Further, a classification 
of defect extensions of prime degree of valued fields that was earlier given only for 
the equicharacteristic case is generalized to the case of mixed characteristic by a
unified definition that works simultaneously for both cases. It
is shown that deeply ramified fields and the other valued fields we introduce only 
admit one of the two types of defect extensions, namely the ones that appear to be 
more harmless in open problems such as local uniformization and the model theory of
valued fields in positive characteristic. We use our knowledge about such defect 
extensions to give a new, valuation theoretic proof of the fact that algebraic 
extensions of deeply ramified fields are again deeply ramified. We also prove finite 
descent, and under certain conditions even infinite descent, for deeply ramified fields. 
These results are also proved for two other related classes of valued fields.
The classes of valued
fields under consideration can be seen as generalizations of the class of tame valued
fields. Our paper supports the hope that it will be possible to generalize to deeply
ramified fields several important results that have been proven for tame fields and 
were at the core of partial solutions of the two open problems mentioned above.
\end{abstract}


\maketitle
%
%
\section{Introduction}
The main topics of this paper are the defect of valued field extensions, which lies 
at the heart of longstanding open problems in algebraic geometry and model theoretic
algebra, and the valuation theory of deeply ramified fields. By studying the latter 
in depth, we will exhibit the connection with the former. On the one hand, this enables 
us to better understand deeply ramified fields, and on the other hand, it shows us 
a possible direction in our attempt to tame the defect.

Our interest in the defect owes its existence to the following well known deep open
problems in positive characteristic:
\sn
1) resolution of singularities in arbitrary dimension,
\sn
2) decidability of the field $\F_q((t))$ of Laurent series over a finite field 
$\F_q$, and of its perfect hull.
\sn
Both problems are connected with the structure theory of valued function fields of
positive characteristic $p$. The main obstruction here is the phenomenon of the
\bfind{defect}, which we will define now.

By $(L|K,v)$ we denote a field extension $L|K$ where $v$ is a valuation
on $L$ and $K$ is endowed with the restriction of $v$. The valuation
ring of $v$ on $L$ will be denoted by $\cO_L\,$, and that on $K$ by
$\cO_K\,$. Similarly, $\cM_L$ and $\cM_K$ denote the valuation ideals of $L$ and $K$.
The value group of the valued field $(L,v)$ will be denoted by $vL$, and
its residue field by $Lv$. The value of an element $a$ will be denoted by $va$, and 
its residue by $av$.

We will say that a valued field extension $(L|K,v)$ is \bfind{unibranched} if the
extension of $v$ from $K$ to $L$ is unique. Note that a unibranched extension is
automatically algebraic, since every transcendental extension always admits several
extensions of the valuation.

If $(L|K,v)$ is a finite unibranched extension, then by the Lemma of Ostrowski,
\begin{equation}                    \label{feuniq}
[L:K]\>=\> \tilde{p}^{\nu }\cdot(vL:vK)[Lv:Kv]\>,
\end{equation}
where $\nu$ is a non-negative integer and $\tilde{p}$ the
\bfind{characteristic exponent} of $Kv$, that is, $\tilde{p}=\chara Kv$ if it is 
positive and $\tilde{p}=1$ otherwise. The factor $d(L|K,v):=\tilde{p}^{\nu }$ is 
the \bfind{defect} of the extension $(L|K,v)$. We call $(L|K,v)$ a \bfind{defect
extension} if $d(L|K,v) >1$, and a \bfind{defectless extension} if $d(L|K,v)=1$.
Nontrivial defect only appears when $\chara Kv=p>0$, in which case $\tilde{p}=p$.
A henselian field $(K,v)$ is called a \bfind{defectless field} if all of its finite
extensions are defectless.

Throughout this paper, when we talk of a \bfind{defect extension $(L|K,v)$ of prime
degree}, we will always tacitly assume that it is a unibranched extension. Then it
follows from (\ref{feuniq}) that $[L:K]=p=\chara Kv$ and that $(vL:vK)=1=[Lv:Kv]$; 
the latter means that $(L|K,v)$ is an \bfind{immediate extension}, i.e.,
the canonical embeddings $vK\hookrightarrow vL$ and $Kv\hookrightarrow Lv$ are onto.

Via ramification theory, the study of defect extensions can be reduced to the study 
of purely inseparable extensions and of Galois extensions of degree $p=\chara Kv$. 
To this end, we fix an extension of $v$ from $K$ to its algebraic closure
$\tilde{K}$. We denote the separable-algbraic closure of $K$ by $K\sep$. The
\bfind{absolute ramification field of $(K,v)$} (with respect to the chosen extension 
of $v$), denoted by $(K^r,v)$, is the ramification field of the normal extension 
$(K\sep|K,v)$. If $a\in\tilde{K}$ such that $(K(a)|K,v)$ is a defect extension, then 
$(K^r(a)|K^r,v)$ is a defect
extension with the same defect (see Proposition~\ref{K(a)K^r(a)}). On the other hand,
$K\sep|K^r$ is a $p$-extension, so $K^r(a)|K^r$ is a tower of
purely inseparable extensions and Galois extensions of degree $p$.

Galois defect extensions of degree $p$ of valued fields of characteristic $p>0$ 
(valued fields of \bfind{equal characteristic}) have been classified by the first 
author in \cite{[Ku6]}. There the extension is said to have
\bfind{dependent} defect if it is related to a purely inseparable defect extension of 
degree $p$ in a way that we will explain in Section~\ref{sectASde}, and to have
\bfind{independent} defect otherwise. Note that the condition
for the defect to be dependent implies that the purely inseparable defect extension 
does not lie in the completion of $(K,v)$, hence if $(K,v)$ lies dense in its perfect
hull (with respect to the topology induced by the valuation), then it cannot have Galois 
defect extensions of prime degree with dependent defect.

The classification of defect extensions is important because work by M.~Temkin (see 
e.g.\ \cite{[Te]}) and by the first author indicates that dependent defect appears 
to be more harmful to the above cited problems than independent defect. Also results 
in the present paper point in this direction; see the discussion in Remark~\ref{roughly}.

An analogous classification of Galois defect extensions of degree $p$ of valued fields 
of characteristic $0$ with residue fields of characteristic $p>0$ (valued fields of
\bfind{mixed characteristic}) has so far not been given. But such a classification is
important for instance for the study of infinite algebraic extensions of
the field $\Qp$ of $p$-adic numbers, which in contrast to $\Qp$ itself may well admit
defect extensions. Indeed, $\Q_p^{ab}$, the maximal abelian extension of $\Qp$, is 
such a field. Other examples will be given in a subsequent paper \cite{KR}. Moreover, 
we wish to study the valuation theory of deeply ramified fields (such as
$\Q_p^{ab}$), which will be introduced below, in full generality without restriction 
to the equal characteristic case. For these fields in particular it is important to
work out the similarities between the equal and the mixed characteristic cases.

\pars
The obvious problem for the definition of ``dependent defect'' in the mixed
characteristic case is that a field of
characteristic $0$ has no nontrivial inseparable extensions. However,
there is a characterization of independent defect equivalent to the one given in
\cite{[Ku6]} that readily works also in the mixed characteristic case, and we use 
it to give a unified definition, as follows. Take a Galois
defect extension $\cE=(L|K,v)$ of prime degree $p$. For every $\sigma$ in its Galois 
group $\Gal (L|K)$, with $\sigma\ne\,$id, we set
\begin{equation}                        \label{Sigsig}
\Sigma_\sigma\>:=\> \left\{ v\left( \left.\frac{\sigma f-f}{f}\right) \right| \, 
f\in L^{\times} \right\} \>.
\end{equation}
This set is a final segment of $vK$ and independent of the choice of $\sigma$ 
(see Theorems \ref{dist_ext_p} and \ref{dist_galois_p}); we denote it by 
$\Sigma_\cE\,$. We will show that it is the unique ramification jump of $\cE$ 
and that $I_{\cE}:= \{a\in L\mid va\in\Sigma_\cE\}$ is the unique ramification ideal 
of $\cE$ (for definitions, see Section~\ref{secthrg}). We will explicitly compute 
$\Sigma_\cE$ and $I_{\cE}$ in Section~\ref{sectdepd}.

We say that $\cE$ has \bfind{independent defect} if
\begin{equation}                                     \label{indepdef}
\left\{\begin{array}{lcr}
\Sigma_{\cE}\!\!&=&\!\!\! \{\alpha\in vK\mid \alpha >H_\cE\}\>\mbox{ for some proper 
convex subgroup $H_\cE$}\\
&&\!\!\! \mbox{ of $vK$ such that $vK/H_\cE$ has no smallest positive element;}
\end{array}\right.
\end{equation}
otherwise we will say that $\cE$ has \bfind{dependent defect}. If $(K,v)$ has 
rank 1 (i.e., its value group is order isomorphic to a subgroup of $\R$), then condition
(\ref{indepdef}) just means that $\Sigma_{\cE}$ consists of all positive elements in $vK$.

\pars
That our definition of ``independent defect'' in mixed characteristic is the right 
one is supported by the following observation. Take a valued field of positive
characteristic. If it lies dense in its perfect hull, then by what we have said before, 
all Galois defect extensions must have independent defect. If in addition the field is
complete and of rank 1, then it is a
perfectoid field. What about perfectoid fields of mixed characteristic? They share with 
their tilts, which are perfectoid fields of positive characteristic, isomorphic absolute 
Galois groups. Hence we expect that also perfectoid fields in mixed characteristic admit 
only independent defects. This indeed holds with our definition. Similarly, the 
Fontaine-Wintenberger Theorem states that the fields $\Q_p(p^{1/p^n}\mid n\in\N)$ and 
$\F_p((t))(t^{1/p^n}\mid n\in\N)$ have isomorphic absolute Galois groups. Both are deeply
ramified (and even semitame) fields (definitions are given below), and as such admit 
only independent defects, as we will show in Theorem~\ref{KEindep}.

\pars
For our purposes, the properties of completeness and rank 1 are irrelevant, and we 
prefer to work with a more flexible (and first order axiomatizable) notion. In fact, 
all perfectoid fields are deeply ramified, in the sense of \cite{GR}. Take a valued 
field $(K,v)$ with valuation ring $\cO_K\,$. Choose any extension of $v$ to $K\sep$ 
and denote the valuation ring of $K\sep$ with respect to this extension by 
$\cO_{K\sep}\,$. Then $(K,v)$ is a \bfind{deeply ramified field} if
\begin{equation}                         \label{GRdefdr}
\Omega_{\cO_{K\sep}|\cO_K} \>=\> 0\>,
\end{equation}
where $\Omega_{B|A}$ denotes the module of relative differentials when $A$ is a ring 
and $B$ is an $A$-algebra. This definition does not depend on the chosen extension of 
the valuation from $K$ to $K\sep$.

According to \cite[Theorem~6.6.12 (vi)]{GR}, a nontrivially valued field $(K,v)$ is
deeply ramified if and only if the following conditions hold:
\sn
{\bf (DRvg)} whenever $\Gamma_1\subsetneq\Gamma_2$ are convex subgroups of the value
group $vK$, then $\Gamma_2/\Gamma_1$ is not isomorphic to $\Z$ (that is, no
archimedean component of $vK$ is discrete);
\sn
{\bf (DRvr)} if $\chara Kv=p>0$, then the homomorphism
\begin{equation}                          \label{homOpO}
\cO_{\hat K}/p\cO_{\hat K} \ni x\mapsto x^p\in \cO_{\hat K}/p\cO_{\hat K}
\end{equation}
is surjective, where $\cO_{\hat K}$ denotes the valuation ring of the completion 
$\hat K$ of $(K,v)$.
\sn
Axiom (DRvr) means that modulo $p\cO_{\hat K}$ every element in $\cO_{\hat K}$ is 
a $p$-th power.

\pars
By altering axiom (DRvg) we will now introduce new classes of valued fields, one of 
them containing the class of deeply ramified fields, and one contained in it in the 
case of positive residue characteristic. We will call $(K,v)$ a
\bfind{roughly deeply ramified field}, or in short an \bfind{rdr field}, if it
satisfies axiom (DRvr) together with:
\sn
{\bf (DRvp)} if $\chara Kv=p>0$, then $vp$ is not the smallest positive element 
in the value group $vK$.
\sn
The reason for the choice of this notion will become visible in Proposition~\ref{rdrrk1}
and will be further discussed in Remark~\ref{roughly}.
Note that (DRvg) implies (DRvp).

\pars
If $\chara Kv=p>0$, then (DRvg) certainly holds whenever
$vK$ is divisible by $p$. We will call $(K,v)$ a \bfind{semitame field} if it 
satisfies axiom (DRvr) together with:
\sn
{\bf (DRst)} if $\chara Kv=p>0$, then the value group $vK$ is $p$-divisible.

\pars
We note:
\begin{proposition}                                \label{ax}
The properties (DRvg), (DRvp) and (DRst) are first order axiomatizable in the 
language of valued fields, and so are the classes of semitame, deeply ramified and 
rdr fields of fixed characteristic.
\end{proposition}
We will give the proof of this proposition and of almost all results that we will
describe now in Section~\ref{sectdr}.

Let us mention at this point that it has been conjectured that the elementary theory of
the perfect hull of $\F_p((t))$ is decidable, but no proof has been given so far. As 
a perfect valued field of positive characteristic, it is semitame, and understanding 
its valuation theory and in
particular its defects may lay the basis for a future proof. Mastering the defect has 
already shown to be an efficient tool to prove results on local uniformization and the
model theory of valued fields, as demonstrated in \cite{[KK1], [KK2], [K3], [K7]}.

\parm
The notion of ``semitame field'' is reminiscent of that of ``tame field''. Let us 
recall the definition of ``tame''. For the purpose of this paper we will slightly
generalize the notion of ``tame extension'' as defined in \cite{[K7]} (there, tame
extensions were only defined over henselian fields).
A unibranched extension $(L|K,v)$ will be called \bfind{tame}
if every finite subextension $E|K$ of $L|K$ satisfies the following conditions:
\sn
$\begin{array}{ll}
{\bf (TE1)} & \mbox{The ramification index $(vE:vK)$ is not divisible by $\chara Kv$.}\\
{\bf (TE2)} & \mbox{The residue field extension $Ev|Kv$ is separable.}\\
{\bf (TE3)} & \mbox{The extension $(E|K,v)$ is defectless.}
\end{array}$
\sn
A valued field $(K,v)$ is called a \bfind{tame field} if it is henselian and its 
algebraic closure with the unique extension of the valuation is a tame extension, and a
\bfind{separably tame field} if it is henselian and its separable-algebraic closure is a 
tame extension. The absolute 
ramification field $(K^r,v)$ is the unique maximal tame extension of the henselian field
$(K,v)$ by \cite[Theorem (22.7)]{[En]} (see also \cite[Proposition 4.1]{[K--P--R]}).
Hence a henselian field is tame if and only if its absolute ramification field is 
already algebraically closed; in particular, every tame field is perfect.

In contrast to tame and separably tame fields, we do not require semitame fields to 
be henselian; in this way they become closer to deeply ramified fields. The other
fundamental difference to tame fields is that semitame fields
may admit defect extensions, but as we will see in Theorem~\ref{KEindep} below, only
those with independent defect. This justifies the hope that many of the results that 
have been proved for tame fields and applied to the problems
we have cited in the beginning (see \cite{[K7],Khr}) can be
generalized (at least) to the case of henselian semitame fields.

\pars
All valued fields of residue characteristic 0 are semitame and rdr fields, and they 
are deeply ramified fields if and only if (DRvg) holds. Likewise, all henselian 
valued fields of residue characteristic 0 are tame fields. In the present paper, we are
not interested in the case of residue characteristic 0, so we will always assume that
$\chara Kv=p>0$. We will now summarize the basic facts about the connections between 
the properties we have introduced. The proofs will be provided in Section~\ref{sectdr}.
\begin{theorem}                            \label{connprop}
1) If $(K,v)$ is a nontrivially valued field with $\chara Kv=p>0$, then the following
logical relations between its properties hold:
\begin{center}
tame field $\Rightarrow$ separably tame field $\Rightarrow$ semitame field 
$\Rightarrow$\\
deeply ramified field $\Rightarrow$ rdr field.
\end{center}
\sn
2) For a valued field $(K,v)$ of rank 1 with $\chara Kv=p>0$, the three properties
``semitame field'', ``deeply ramified field'' and ``rdr field'' are equivalent.
\mn
3) For a nontrivially valued field $(K,v)$ of characteristic $p>0$, the following
properties are equivalent:
\n
a) $(K,v)$ is a semitame field,
\n
b) $(K,v)$ is a deeply ramified field,
\n
c) $(K,v)$ is an rdr field,
\n
d) $(K,v)$ satisfies axiom (DRvr),
\n
e) the completion of $(K,v)$ is perfect,
\n
f) $(K,v)$ is dense in its perfect hull,
\n
g) $(K^p,v)$ is dense in $(K,v)$.
\mn
4) Every perfect valued field of positive characteristic is a semitame field.
\end{theorem}

We note that for valued fields of mixed characteristic, axiom (DRvr) can be 
substituted by a version where $\hat K$ is replaced by $K$ (see Lemma~\ref{OK/pOK}), 
and even $p$ can be replaced by  elements of certain lower or higher values (see
Propositions~\ref{dpd'1} and~\ref{dpd'2}).

In \cite{K9} the equivalence of assertions a) and f) of part 3) of this theorem is 
used to show that every valued field of positive characteristic that has only finitely
many Artin-Schreier extensions is a semitame field. This proves that a nontrivially
valued field of positive characteristic that is definable in an NTP$_2$ theory
is a semitame field, as it is shown in \cite{CKS} that such a field has only finitely
many Artin-Schreier extensions.

\parm
Take a valued field $(K,v)$ of characteristic 0 with residue characteristic $p>0$.
Decompose $v=v_0\circ v_p\circ \ovl{v}$, where $v_0$ is the finest coarsening of $v$ 
that has residue characteristic 0, $v_p$ is a rank 1 valuation on $Kv_0\,$, and 
$\ovl{v}$ is the valuation induced by $v$ on the residue field of $v_p$ (which is of
characteristic $p>0$). The valuations $v_0$ and $\ovl{v}$ may be trivial. Further, 
following \cite{J}, we call the value group $vK$ \bfind{roughly $p$-divisible} if 
$v_p\circ \ovl{v}\,(Kv_0)$ (the value group of $v_p\circ \ovl{v}$ on $Kv_0$) is 
$p$-divisible. Then we have:

\begin{proposition}                   \label{rdrrk1}
Under the above assumptions, the following assertions are equivalent:
\sn
a) $(K,v)$ is a roughly deeply ramified field,
\sn
b) $(Kv_0,v_p\circ \ovl{v})$ is a deeply ramified field,
\sn
c) $(Kv_0,v_p)$ is a deeply ramified field,
\sn
d) $(K,v)$ satisfies (DRvr) and $vK$ is roughly $p$-divisible.
\end{proposition}
\n
Note that by part 2) of Theorem~\ref{connprop}, the properties ``semitame field'', 
``deeply ramified field'' and ``rdr field'' are equivalent for $(Kv_0,v_p)$.

\pars
From Theorem~\ref{connprop} and Proposition~\ref{rdrrk1} it can be deduced that the 
three properties ``semitame'', ``deeply ramified'' and ``rdr'' behave well for 
composite valuations.
\begin{proposition}                                 \label{wcow}
Take an arbitrary valued field $(K,v)$ and assume that $v=w\circ\ovl{w}$ with $w$ and 
$\ovl{w}$ nontrivial. Then $(K,v)$ is an rdr field if and only if $(K,w)$ and 
$(Kw,\ovl{w})$ are. If $\chara Kw>0$, then for $(K,v)$ to be an rdr field it suffices 
that $(K,w)$ is an rdr field. The same holds for ``semitame'' and ``deeply ramified'' in
place of ``rdr''.

If $\chara Kw=0$, then for $(K,v)$ to be an rdr field it suffices that $(Kw,\ovl{w})$ 
is an rdr field.
\end{proposition}

\pars
For deeply ramified fields, the first assertion of the next theorem has been proved before
(see \cite[Corollary 6.6.16 (i)]{GR}), based on their definition given in (\ref{GRdefdr}).
\begin{theorem}                               \label{algext}
Every algebraic extension of a deeply ramified field is again deeply ramified. The 
same holds for semitame fields and for rdr fields.
\end{theorem}
\n
We will give the easy proof for the equal characteristic case in
Proposition~\ref{algextpos}.
The proof for the mixed characteristic case can be reduced to the study of Galois 
defect extensions of prime degree via the following theorem:
\begin{theorem}                                     \label{rdrram}
Take a valued field $(K,v)$, fix any extension of $v$ to $\tilde{K}$, and let
$(K^r,v)$ be the corresponding absolute ramification field of $(K,v)$. Then $(K^r,v)$ is 
an rdr field if and only if $(K,v)$ is, and $(K^r,v)$ is a semitame field if and only if 
$(K,v)$ is. If $(K,v)$ is an rdr field, then $(K^r,v)$ is a deeply ramified field.
\end{theorem}
\n
Note that the last assertion holds since if $(K^r,v)$ is an rdr field, then it is 
already a deeply ramified field because $vK^r$ is divisible by every prime distinct 
from the residue characteristic. However, it is not true in general that this implies
that $(K,v)$ is deeply ramified, since (DRvg) always holds in $(K^r,v)$ (as long as 
$v$ is nontrivial), while it may not hold in $(K,v)$.
\begin{corollary}                                    \label{rdrramcor}
1) Take an algebraic (not necessarily finite) extension $(L|K,v)$ of valued fields. 
If $K^r=L^r$ with respect to some extension of $v$ from $L$ to $\tilde{L}$, then  
$(L,v)$ is an rdr field if and only if $(K,v)$ is, and the same holds for 
``semitame'' in place of ``rdr''.
\sn
2) Take a valued field $(K,v)$, fix any extension of $v$ to $\tilde{K}$, and let
$(K^h,v)$ be the henselization of $(K,v)$ in $(\tilde{K},v)$. Then $(K^h,v)$ is 
a deeply ramified field if and only if $(K,v)$ is, and the same holds for ``rdr'' 
and ``semitame'' in place of ``deeply ramified''.
\end{corollary}
\n
Note that the assumption of part 1) holds in particular if $(L|K,v)$ is a tame 
extension. We see that we have infinite descent of the properties ``rdr'' and
``semitame'' through extensions in the absolute ramification field and in particular
through tame extensions. If the lower field already satisfies (DRvg), then the 
descent also works for ``deeply ramified''. For all of the properties, we have
finite descent in general:
\begin{theorem}                           \label{down}
Take a finite extension $(L|K,v)$. If $(L,v)$ is a deeply ramified field, then so 
is $(K,v)$. The same holds for ``rdr'' and ``semitame'' in place of ``deeply 
ramified''.
\end{theorem}

\parm
The next theorem
addresses the connection of the properties we have defined with the classification of 
the defect. Take a valued field $(K,v)$ of residue characteristic $p>0$. If $\chara K=0$,
then we denote by $(vK)_{vp}$ the smallest convex subgroup of $vK$ that contains $vp$;
it is equal to the value group $v_p\circ \ovl{v}\,(Kv_0)$ which we defined earlier.
If $\chara K>0$, then we set $(vK)_{vp}=vK$. If $(K,v)$ is of mixed characteristic, 
then we set $K':=K(\zeta_p)$, where $\zeta_p$ is a primitive $p$-th roots of unity;
otherwise, we set $K':=K$. Then
we call $(K,v)$ an \bfind{independent defect field} if for some extension of $v$ to 
$K'$, all Galois defect extensions of $(K',v)$ of degree $p$ have independent
defect. (This definition does not depend on the chosen extension of $v$ as all 
extensions are conjugate.) We will show in Theorem~\ref{KEindep} that all rdr fields, 
and hence all deeply ramified and semitame fields, are independent defect fields.
\begin{remark}
If $(K,v)$ is an rdr field of mixed characteristic, then it does not necessarily 
contain a primitive $p$-th root of unity. In this case, a condition on Galois 
defect extensions may not contain enough 
information. We also need information on extensions by $p$-th roots which will 
then not be Galois. This is why we pass to the field $K'$ in our definition.
\end{remark}

From Proposition~\ref{rdrrk1} we see that in the case of fields $(K,v)$ of mixed
characteristic, $(Kv_0,v_p)$ is essential for the rdr property. In the theory of 
formally $p$-adic fields (cf.\ \cite{PR}), $Kv_0$ is called the \bfind{core field},
and it is usually considered with the valuation $v_p\circ \ovl{v}$. However, as we have 
just noted, the valuation $v_p$ itself is important, and we will have to work with its
residue field $(Kv_0)v_p=Kv_0 v_p\,$.
In the decomposition $v=v_0\circ v_p\circ \ovl{v}$ the valuation $v_p$ is at the center,
and so we define $\erf (K,v):= Kv_0 v_p$ as one may call it the \bfind{central residue
field}. If $(K,v)$ is of equal characteristic, we set $\erf (K,v):= Kv$.
\begin{theorem}                               \label{KEindep} 
1) Take a valued field $(K,v)$ with $\chara Kv=p>0$. Then $(K,v)$  is an rdr field 
if and only if $(vK)_{vp}$ is $p$-divisible, $\erf (K,v)$ is perfect, and 
$(K,v)$ is an independent defect field.
\sn
2) A nontrivially valued field $(K,v)$ is semitame if and only if every unibranched
Galois extension of $(K',v)$ of prime degree is either 
tame or an extension with independent defect.
\end{theorem}

\begin{remark}                         \label{roughly}
Let us consider a field $(K,v)$ of mixed characteristic. Then $(K,v_0)$ is a 
field of residue characteristic $0$. If $(K,v)$ is henselian, then so is $(K,v_0)$, and  
it is a defectless field and satisfies strong model theoretic principles (e.g.\ 
completeness, model completeness and decidability relative to their value groups and 
residue fields). Recently, the idea appeared in the literature that in order for $(K,v)$
to have such good properties, one only has to ask that the core field satisfies suitable
conditions. For example, \cite{HH} deals with the generalization of model theoretic
properties from the class of algebraically maximal Kaplansky fields to the more general
class of henselian fields whose core fields are algebraically maximal Kaplansky fields.
A valued field is called \bfind{algebraically maximal} (respectively, 
\bfind{separable-algebraically maximal}) if it admits no nontrivial immediate algebraic
(respectively, separ\-able-algebraic) extensions. Since henselizations are immediate 
separ\-able-algebraic extensions, every separable-algebraically maximal field is 
henselian. If the core field of $(K,v)$ is a Kaplansky field, then we may call $(K,v)$ a 
\bfind{roughly Kaplansky field}, as the passage from Kaplansky field to roughly 
Kaplansky field is nothing but the passage from $p$-divisible value group to roughly
$p$-divisible value group. 

Similarly, $(K,v)$ is a defectless field if and only if its core field is. Let us call
$(K,v)$ a \bfind{roughly tame field} if it is henselian and its core field is a tame 
field; again 
the generalization consists in replacing ``$p$-divisible value group'' by ``roughly
$p$-divisible value group''. It is shown in \cite{Rz} that a henselian field $(K,v)$ is 
roughly tame if and only if all of its algebraic extensions are defectless fields. 
The reader may note that in general, infinite algebraic extensions of defectless fields 
may not again be defectless fields.

We have chosen the name ``roughly deeply ramified field'' since in this case, the 
generalization of the notion ``deeply ramified field'' consists in replacing 
condition (DRvg) by condition ``roughly $p$-divisible value group'', which removes any condition on the value group $v_0 K$. It should be noted 
that the value group of every rdr and hence also every deeply ramified field is 
roughly $p$-divisible (see Lemma~\ref{basprop2}). From Theorem~\ref{algext} together
with part 1) of Theorem~\ref{KEindep} we know that for every rdr field, in analogy to 
the case of roughly tame fields, every algebraic extension is an independent defect 
field. At this point, we do not know whether the latter property characterizes the rdr
fields. In any case, it appears to be unlikely that a similar result can be shown with
``independent'' replaced by ``dependent''; if this is indeed impossible, then it is 
another indication that independent defect is more harmless than dependent defect.
\end{remark}

\parm
The classification of Galois defect extensions of prime degree in the equal
characteristic case is also an important tool in the proof of Theorem~1.2 of
\cite{[Ku6]}, which we will state now. 
\begin{theorem}                                \label{chardl}
A valued field of positive characteristic is a henselian and defectless field if 
and only if it is separable-algebraically maximal and each finite purely inseparable
extension is defectless.
\end{theorem}

This theorem in turn is used in \cite{[K1]} for the construction of an example 
showing that a certain natural axiom system for the elementary theory of $\F_p((t))$
(``henselian defectless valued field of characteristic $p$
with residue field $\F_p$ and value group a $\Z$-group'') is not complete.

A full analogue of Theorem~\ref{chardl} in mixed characteristic is not presently known.
However, in a subsequent paper \cite{KR} we will show:
\begin{theorem}                                \label{chardlrdr}
Every algebraically maximal rdr field is a perfect, henselian and defectless field.
\end{theorem}

The study of mixed characteristic independent defect fields that are not rdr fields
is only at its infancy. We hope that the valuation theoretic proof of
Theorem~\ref{algext} will be a basis for further insight.
At this point, we are able to prove:
\begin{proposition}                                    \label{idf}
1) If $(K^r,v)$ is an independent defect field, then so is $(K,v)$.
\sn
2) A valued field $(K,v)$ of equal positive characteristic is an independent 
defect field if and only if
every immediate purely inseparable extension of $(K,v)$ lies in its completion.
\end{proposition}

\pars
It is an important fact that the properties of valued fields of being henselian, tame, 
semitame, deeply ramified or rdr all are preserved under infinite algebraic extensions.
In contrast to this, the properties of being a defectless or an independent defect field 
are not necessarily preserved, as will be shown in 
\cite{KR} by the construction of a suitable algebraic extension of $\Q_p\,$. However, we
conjecture that if $(K,v)$ is an independent defect field, then so is $(K^r,v)$.

\parm
Continuing the work presented in \cite{[CP]}, the idea has been suggested to
employ higher ramification groups for the study of the ramification theory of 
2-dimensional valued function fields. When working over valued fields with arbitrary
value groups, the classical ramification numbers have to be replaced by
\bfind{ramification jumps} which can be understood as cuts (or equivalently, final
segments) in the value group (cf.\ Section~\ref{secthrg}). We will characterize 
Galois defect extensions $\cE$ of prime degree having independent defect via their
ramification jumps $\Sigma_\cE$ and their associated ramification ideals $I_\cE\,$. 
In Section~\ref{sectGaldefdegp} we will prove:
\begin{theorem}                              \label{eqindep}
Take a Galois defect extension $\cE=(K(a)|K,v)$ of prime degree.
Then $\Sigma_{\cE}$ is the unique ramification jump of $\cE$, $I_\cE$ is the unique
ramification ideal of $\cE$. If in addition $(K,v)$ satisfies (DRvg), then the following
assertions are equivalent:
\sn
a) $\cE$ has independent defect,
\sn
b) the ramification jump of $\cE$ is equal to $\{\alpha\in vK\mid \alpha >H_\cE\}$
for some proper convex subgroup $H_\cE$ of $vK$,
\sn
c) the ramification ideal of $\cE$ is a nontrivial prime ideal of $\cO_L$.
\sn
If assertion c) holds, then the localization of $\cO_L$ with respect to the ramification
ideal is a valuation ring on $L$ containing
$\cO_L$, i.e., the associated valuation is a coarsening of $v$.
\sn
If the rank of $(K,v)$ is 1, then $H_\cE$ can only be equal to $\{0\}$ and if the
ramification ideal is a prime ideal, then it can only be equal to $\cM_L\,$.
\end{theorem}

More equivalent conditions (without the assumption of (DRvg)) will be given in a 
subsequent paper.

\bn
%
%
%
\section{Preliminaries}                              \label{sectprel}
\subsection{Cuts, distances and defect}                      \label{sectcdist}
\mbox{ }\sn
We recall basic notions and facts connected with cuts in ordered abelian groups and 
distances of elements of valued field extensions. For the details and proofs see 
Section 2.3 of \cite{[Ku6]} and Section 3 of \cite{[KV]}.

Take a totally ordered set $(T,<)$. For a nonempty subset $S$ of $T$ and an element 
$t\in T$ we will write $S<t$ if $s<t$ for every $s\in S$. A set $S\subseteq T$ is called 
an \bfind{initial segment} of $T$ if for each
$s\in S$ every $t <s$ also lies in $S$. Similarly, $S\subseteq T$ is called
a \bfind{final segment} of $T$ if for each $s\in S$ every $t >s$ also lies in $S$.
A pair $(\Lambda^L,\Lambda^R)$  of subsets of $T$ is called a \bfind{cut} in $T$ if 
$\Lambda^L$ is an initial segment of $T$ and $\Lambda^R=T\setminus \Lambda^L$; it then 
follows that $\Lambda^R$ is a final segment of $T$. To compare cuts in  $(T,<)$ we will 
use the lower cut sets comparison. That is, for two cuts
$\Lambda_1=(\Lambda_1^L,\Lambda_1^R) ,\, \Lambda_2=(\Lambda_2^L,\Lambda_2^R)$ in $T$ we 
will write $\Lambda_1<\Lambda_2$ if $\Lambda^L_1\varsubsetneq\Lambda^L_2$, and $\Lambda_1
\leq \Lambda_2$ if $\Lambda^L_1 \subseteq  \Lambda^L_2$.

For any $s\in T$ define the following \bfind{principal cuts}:
\[
s^- \,:=\, (\{t\in T \,| \, t< s\}, \{t\in T \,|\, t\geq s\})\>,\;
s^+ \,:=\, (\{t\in T \,| \, t\leq s\}, \{t\in T \,|\, t>s\})\>.
\]
We identify the element $s$ with $s^+$. Therefore, for a cut 
$\Lambda=(\Lambda^L,\Lambda^R)$ in $T$ and an element $s\in T$ the inequality 
$\Lambda <s$ means that for every element $t\in \Lambda^L$ we have $t<s$. Similarly, 
for any subset $M$ of $T$ we define $M^+$ to be 
a cut $(\Lambda^L,\Lambda^R)$ in $T$ such that
$\Lambda^L$ is the smallest initial segment containing $M$, that is,
\[
M^+\>=\> (\{t\in T\,|\, \exists\, m\in M \,\,t\leq m\}, \{t\in T\,|\, t>M\} )\>.
\]
Likewise, we denote by $M^-$ the cut $(\Lambda^L,\Lambda^R)$ in $T$ such that $\Lambda^L$ 
is the largest initial segment disjoint from $M$, i.e.,
\[
M^-\>=\> (\{t\in T\,|\, t<M\}, \{t\in T\,|\, \exists\, m\in M \,\,t\geq m\} )\>.
\]

\parm
When $M$ is a subset of the value group $vK$ of a valued field $(K,v)$, we will 
{\bf always denote by $M^-$ and $M^+$ the cuts induced in the divisible hull 
$\widetilde{vK}$ of $vK\>$.} We note that if $H$ is a subgroup of $vK$, then its convex 
hull in $\widetilde{vK}$ is a convex subgroup of $\widetilde{vK}$. In connection with 
our definition of independent defect, we will need the following fact:
\begin{lemma}                        \label{S^+=H^-}
Take an initial segment $\Sigma$ of $vK$ and a proper convex subgroup $H$ of $vK$. Then 
\[
\Sigma^+\>=\>H^-
\]
(as cuts in $\widetilde{vK}$) if and only if $\,-\Sigma=\{\alpha\in vK\mid \alpha >H\}$ 
and $vK/H$ has no smallest positive element.
\end{lemma}
\begin{proof}
$\Rightarrow$: Suppose that $vK/H$ has a smallest positive element $\alpha+H$, $\alpha
\in vK$. Then the element $-\alpha/2\in\widetilde{vK}$ is larger than all elements of 
$\Sigma$, but does not lie in the convex hull of $H$ in $\widetilde{vK}$. Therefore, 
$\Sigma^+<\alpha/2<H^-$, showing that $\Sigma^+ \ne H^-$.
\sn
$\Leftarrow$: From $-\Sigma=\{\alpha\in vK\mid \alpha >H\}$ it follows that $\Sigma^+ 
\leq H^-$. If equality does not hold, then there is $\beta\in\widetilde{vK}$ such that 
$\Sigma^+ < \beta < H^-$. Take $n\in\N$ such that $n\beta\in vK$. 
As $vK/H$ has no smallest positive element, the positive elements cannot be bounded away
from $0$ by an element in the divisible hull of $vK/H$, and so there must be $\gamma\in
vK$ such that $0<\gamma+H<\frac 1 n (-n\beta+H)$. This implies $n\gamma +H<-n\beta+H$,
whence $n\gamma<-n\beta$ in $vK$ and $-\gamma>\beta$ in $\widetilde{vK}$. This implies 
that $-\gamma>\Sigma^+$. But $0<\gamma+H$ means that $-\gamma<0$ does not lie in $H$, so 
it lies in $\Sigma$, contradiction. Thus, $\Sigma^+=H^-$ must hold.
\end{proof}

\pars
For every extension $(L|K,v)$ of valued fields and $z\in L$ define
\[
v(z-K)\>:=\>\{v(z-c)\, |\, c\in K\}\>.
\]
The set $v(z-K)\cap vK$ is an initial segment of $vK$ and thus the lower cut set of 
\mbox{a cut} in $vK$. However, it is more convenient to work with the cut
\[
\dist (z,K)\>:=\>(v(z-K)\cap vK)^+\;\; \textrm{ in the divisible hull $\widetilde{vK}$ 
of $vK\>$.}
\]

We call this cut the \bfind{distance of $z$ from $K$}. The lower cut set of $\dist (z,K)$ 
is the smallest initial segment of $\widetilde{vK}$ containing $v(z-K) \cap vK$. If 
$(F|K,v)$ is an algebraic subextension of $(L|K,v)$ then $\widetilde{vF}=\widetilde{vK}$.
Thus $\dist (z,K)$ and $\dist (z,F)$ are cuts in the same group and we can compare these 
cuts by set inclusion of the lower cut sets.
Since $v(z-K)\subseteq v(z-F)$ we deduce that
\[
\dist (z,K)\> \leq  \>\dist (z,F)\>.
\]

If $\chara K=p>0$ and $z\in K$, then $K^p$ is a subfield of $K$, and the expressions
$v(z-K^p)$ and $\dist(z,K^p)$
are covered by our above definitions. We generalize this to the case where $\chara K=0$ 
with the same definitions but note that $v(z-K^p)\cap vK$ is not necessarily an initial 
segment of $vK$.

\pars
If $y$ is another element of $L$, then we define:
\[
z\sim_K y\;:\Leftrightarrow\; v(z-y)\> >\> \dist (z,K)\>.
\]
The next lemma shows, among other things, that the relation $\sim_K$ is symmetrical.
\begin{lemma}                                     \label{relation}
Take a valued field extension $(L|K,v)$ and elements $z,y\in L$. 
\sn
1) If $z\sim_K y$, then $v(z-c)=v(y-c)$ for all $c\in K$ such that $v(z-c)\in vK$,  
$v(z-K)=v(y-K)$, $\dist (z,K)=\dist (y,K)$, and $y\sim_K z$.
\sn
2) If $(K(z)|K,v)$ is immediate, then $v(z-K)$ is a subset of $vK$ without largest 
element.
\end{lemma}
\begin{proof}
1): This is part (1) of Lemma 2.17 in \cite{[Ku6]}.
\sn
2): It follows from \cite[Theorem 1]{[Ka]} that $v(z-K)$ has no largest element. To prove 
that it is a subset of $vK$, take $c\in K$; we wish to show that $v(z-c)\in vK$. Choose 
$d\in K$ such that $v(z-d)>v(z-c)$. Then $v(z-c)=\min\{v(z-c),v(z-d)\}=v(c-d)\in vK$.
\end{proof}

For any $\alpha \in vK$ and each cut $\Lambda $ in $vK$ we set $\alpha +\Lambda 
:=(\alpha +\Lambda^L, \alpha+\Lambda^R)$. \mbox{An immediate} consequence of the 
above definitions is the following lemma:
\begin{lemma}                                          \label{dist}
Take an extension $(L|K,v)$ of valued fields. Then for every element $c\in K$ and 
$y,z\in L$,
\sn
1) $\dist (z+c,K)=\dist (z,K)$,
\sn
2) $\dist (cz,K)=vc+\dist (z,K)$.
\end{lemma}

Here are some important properties of distances in valued field
extensions. For the proof of the next lemma see \cite[Lemma 7]{[BK3]} and 
\cite[Lemma 2.5]{[Ku6]}.
\begin{lemma}                        \label{ivd} 
Take any immediate extension $(F|K,v)$ and a finite defectless unibranched
extension $(L|K,v)$.
Then the extension of $v$ from $F$ to $F.L$ is unique, $(F.L|F,v)$ is
defectless, $(F.L|L,v)$ is immediate, and for every $a\in F\setminus K$ we have: 
\[
\dist(a,K)=\dist(a,L).
\]
Moreover, $F|K$ and $L|K$ are linearly disjoint, i.e.,
\[
[F.L:F] \>=\> [L:K]\>.
\]
\end{lemma}

\pars
For the proof of the following results see \cite[Lemmas 5 and 9]{[BK3]}.
\begin{lemma}	                                           \label{dist_hens}
Take a unibranched extension $(F|K,v)$ and an extension of
$v$ to the algebraic closure of $F$. Take $K^h$ to be the henselization of $K$ with
respect to this fixed extension
of $v$. Then for every $a\in F$ we have that $[K(a):K]=[K^h(a):K^h]$ as well as
\[
d(K(a)|K,v)\>=\>d(K^h(a)|K^h,v) \quad \mbox{ and }\quad \dist(a,K)\>=\>\dist(a,K^h)\>.
\]
\end{lemma}

For the proof of the next proposition, see \cite{[Ku6]}, Proposition 2.8.
\begin{proposition}			\label{tamedef}
Take a henselian field $(K,v)$ and a tame extension $N$ of $K$. Then for any finite 
extension $L|K$,
\[
d(L|K,v) \>=\> d(L.N|N,v)\>.
\]
In particular, $(K,v)$ is defectless 
if and only if $(K^r,v)$ is defectless. 
\end{proposition}

\pars
For the following theorem, see \cite[Theorem 1]{[Ka]} and \cite[Theorem 2.19]{[Ku6]}.
\begin{theorem}                                              \label{CutImm}
If $(L|K,v)$ is an immediate extension of valued fields, then for every element 
$a\in L\setminus K$ the set $v(a-K)$ is an initial segment of $vK$ without maximal 
element. In particular, $va<\dist (a,K)$.
\end{theorem}

The following partial converse of this theorem also holds 
(see~\cite[Lemma 4.1]{[Bl3]}, cf. also~\cite[Lemma 2.21]{[Ku6]}):
\begin{lemma}                                   \label{imm_deg_p}
Assume that $(K(a)|K,v)$ is a unibranched extension of prime degree such 
that $v(a-K)$ has no maximal element. Then the
extension $(K(a)|K,v)$ is immediate and hence a defect extension.
\end{lemma}

The property that the set $v(a-K)$ has no maximal element does not in general imply 
that $(K(a)|K,v)$ is immediate. However, the next lemma (see e.g.~\cite[Lemma 2.1]
{[KV]}) shows that if in addition $(K,v)$ is henselian
and $a$ is algebraic over $K$, then $(K(a)|K,v)$ is a defect extension.
\begin{lemma}                                         \label{dist_defectless}
If $(K,v)$ is henselian and $(L|K,v)$ is a finite defectless extension, then for every
element $a\in L$ the set $v(a-K)$ admits a maximal element.
\end{lemma}

We will need a version of Lemma~\ref{imm_deg_p} that also works for extensions that 
are not assumed to be unibranched.
\begin{lemma}                                   \label{ude}
Assume that $(K(a)|K,v)$ is an extension of degree at most $p=\chara Kv$
of rank $1$ valued fields  such that $v(a-K)$ has no maximal element but is 
bounded from above in $vK$. Then the
extension $(K(a)|K,v)$ is a unibranched defect extension.
\end{lemma}
\begin{proof}
Take a henselization $(K^h,v)$ and consider the extension $(K^h(a)|K^h,v)$ which again 
is of degree at most $p$. Take any 
$b\in K^h$. Since $K^h(a)|K$ is algebraic, we know that $vK^h(a)$ lies in the 
divisible hull $\widetilde{vK}$ of $vK$ and thus there is some $\alpha\in vK$ such that 
$\alpha>v(a-b)$. Since $(K,v)$ is of rank $1$ by assumption, $(K,v)$ lies dense in 
$(K^h,v)$ (cf.\ \cite[32.11 and 32.18]{[W]}). Therefore, there is some $c\in K$ such 
that $v(b-c)\geq \alpha>v(a-b)$, so that $v(a-b)=v(a-c)\in v(a-K)$. This shows that 
$v(a-K^h)=v(a-K)$. Hence also  $v(a-K^h)$ has no maximal element and is 
bounded from above in $vK=vK^h$. Thus in particular $(K^h(a)|K^h,v)$ is a nontrivial
extension. 

Suppose the extension $(K^h(a)|K^h,v)$ were defectless. Then by 
Lemma~\ref{dist_defectless} the set $v(a-K^h)$ would admit a maximal element, a
contradiction. This shows that $(K^h(a)|K^h,v)$ has nontrivial defect $d(K^h(a)|K^h,v)$. 
By the Lemma of Ostrowski, 
\[
[K^h(a):K^h]= d(K^h(a)|K^h,v)(vK^h(a):vK^h)[K^h(a)v:K^hv],
\]
where $d(K^h(a)|K^h,v)=p^n$ for some $n\geq 1$. Since $[K^h(a):K^h]\leq p$, we deduce that
\[
[K^h(a):K^h]=p =d(K^h(a)|K^h,v).
\]
Hence $K(a)|K$ is linearly disjoint from $K^h|K$ and thus it is a unibranched extension 
(cf. Lemma~2.1 of~\cite{[BK2]}). Moreover, from Lemma~\ref{dist_hens} we deduce that 
$(K(a)|K,v)$ is a defect extension. 
\end{proof}

\pars
The next lemma follows from \cite[Lemma~8]{[Ka]} and \cite[Lemma~5.2]{[KV]}. Note that if 
$(K(a)|K,v)$ is an extension such that $v(a-K)$ has no maximal element, then by the proof 
of \cite[Theorem 1]{[Ka]}, $a$ is limit of a pseudo Cauchy sequence in $(K,v)$ without 
limit in $K$, and by \cite[part a) of Lemma~4.1]{[KV]} its approximation type over 
$(K,v)$ is immediate. We use the Taylor expansion
\begin{equation}                             \label{Taylorexp}
f(X) = \sum_{i=0}^{n} \partial_i f(c) (X-c)^i
\end{equation}
where $\partial_i f$ denotes the $i$-th \bfind{Hasse-Schmidt derivative} of $f$.
\begin{lemma}                                \label{Kap}
Take a nontrivial extension $(K(a)|K,v)$ of degree $p$. Assume that $v(a-K)$ has no
maximal element. Then for every nonconstant polynomial $f\in K[X]$ of degree $<p$ 
there is some $\gamma\in v(a-K)$ such that for all $c\in K$ with $v(a-c)\geq \gamma$ 
and all $i$ with $1\leq i\leq\deg f$, we have:
\n
the values $v\partial_i f(c)$ are fixed, equal to $v\partial_i f(a)$, \n
the values $v\partial_i f(c)+i\cdot v(a-c)$ are pairwise distinct,
\begin{eqnarray}                                   
v\partial_1 f(c) + v(a-c) &<& v\partial_i f(c)+i\cdot v(a-c) 
\>\mbox{ whenever }i\ne 1, \label{xf-}\\    
v(f(a)-f(c)) &=& v\partial_1 f(c) + v(a-c)\>,\>\mbox{ and} \label{Kapvfi} \\
\dist(f(a),K) &=& v\partial_1 f(c) + \dist(a,K)\>. \label{Kapvfi2}
\end{eqnarray}
\end{lemma}

The following is Lemma 2.4 of~\cite{[Ku6]}.
\begin{lemma}               \label{defc}
Take a valued field $(K,v)$, a finite extension $(L|K,v)$ and a
coarsening $w$ of $v$ on $L$. If $(K,v)$ is henselian, then so is
$(K,w)$. If $(L|K,v)$ is defectless, then so is $(L|K,w)$.
\end{lemma}

\subsection{The absolute ramification field}
\begin{proposition}                           \label{K(a)K^r(a)}
Take an immediate unibranched extension $(K(a)|K,v)$. Extend $v$ to the algebraic 
closure of $K$ and let $(K^h,v)$ be the henselization and $(K^r,v)$ the absolute
ramification field of $(K,v)$ with respect to this extension. Then the extension
$(K^r(a)|K^r,v)$ is immediate with
\begin{eqnarray}
[K^r(a):K^r] &=& [K^h(a):K^h] \>=\> [K(a):K]\>,   \label{degKK^r}  \\
d(K^r(a)|K^r,v) &=& d(K^h(a)|K^h,v) \>=\> d(K(a)|K,v)\>,   \label{defKK^r}\\
\dist(a,K^r) &=& \dist(a,K^h) \>=\> \dist(a,K)\>.  \label{distKK^r}
\end{eqnarray}
If $(N|K,v)$ is any subextension of $(K^r|K,v)$, then $[N(a):N]=[K(a):K]$ and
\begin{equation}                       \label{distKN}
d(N(a)|N,v)\>=\>d(K(a)|K,v)\;\mbox{ and }\;\dist(a,N)\>=\>\dist(a,K)\>.
\end{equation}
\end{proposition}
\begin{proof}
Since $(K(a)|K,v)$ is a unibranched extension, we know from Lemma~\ref{dist_hens} 
that $[K^h(a):K^h]=[K(a):K]$ as well as $d(K^h(a)|K^h,v)=d(K(a)|K,v)$ and
$\dist(a,K^h)=\dist(a,K)$. Since $(K(a)|K,v)$ is also immediate,
\[
[K^h(a):K^h]\>=\>[K(a):K] \>=\> d(K(a)|K,v) \>=\> d(K^h(a)|K^h,v)\>,
\]
showing that also $(K^h(a)|K^h,v)$ is immediate.

Further, $(K^r|K^h,v)$ is tame and hence the union of its finite defectless 
subextensions. Thus by Lemma~\ref{ivd}, $(K^r(a)|K^r,v)$ is immediate with 
$[K^r(a):K^r]=[K^h(a):K^h]$ and $\dist(a,K^r)=\dist(a,K^h)$.
By Proposition~\ref{tamedef}, $d(K^r(a)|K^r,v) = d(K^h(a)|K^h,v)$.

\parm
Finally, if $(N|K,v)$ is a subextension of $(K^r|K,v)$, then $N^r=K^r$. Hence by 
(\ref{degKK^r}), $[N(a):N]=[N^r(a):N^r]=[K^r(a):K^r]=[K(a):K]$,
by (\ref{defKK^r}), $d(N(a)|N,v)=d(N^r(a)|N^r,v)=d(K^r(a)|K^r,v)=d(K(a)|K,v)$,
and by (\ref{distKK^r}), $\dist(a,N)=\dist(a,N^r)=\dist(a,K^r)=\dist(a,K)$.
\end{proof}

\pars
For the proof of the following results, see Lemma 2.9 of~\cite{[Ku6]}.
\begin{lemma}			        	\label{ramfield}
Take any valued field $(K,v)$ and let $K^h$ and $K^r$ be its henselization and its
absolute ramification field with respect to any extension of $v$ to the algebraic 
closure of $K$. If $\chara Kv=0$, then $K^r$ is algebraically
closed. If $\chara Kv=p>0$, then every finite extension of $K^r$ is a tower of normal
extensions of degree~$p$. Further, if $L|K$ is a finite extension, then there is
already a finite tame extension $N$ of $K^h$ such that $L.N|N$ is such a tower.
\end{lemma}

The proof of this lemma uses the fact that if $\chara Kv=p>0$, then $K\sep|K^r$ is a 
$p$-extension. From this we can also conclude:
\begin{corollary}                                \label{Krnirou}
Every absolute ramification field contains all $p$-th roots of unity.
\end{corollary}

\pars 
Finally, we will need the following fact:
\begin{lemma}                                  \label{Krw|Kwsep}
Let $(K^r,v)$ be the absolute ramification field of $(K,v)$, and assume that 
$v=w\circ \ovl{w}$. Then the extension $K^rw|Kw$ is separable.
\end{lemma}
\begin{proof}
For a detailed proof, see the chapter on ramification theory in \cite{[K6]}. For the
convenience of the reader, we give 
here a sketch of the proof. We use several facts from ramification theory. The 
assertion is trivial if $\chara Kv=0$, so we let $\chara Kv=p>0$.

\sn
i) If $(K^h,v)$ is the henselization of $(K,v)$, then $(K^hw,\ovl{w})$ is the
henselization of $(Kw,\ovl{w})$. In particular, $K^hw|Kw$ is separable.
\sn
ii) If $(L_1|L_2,w\circ\ovl{w})$ is a finite defectless extension, then so are 
$(L_1|L_2,w)$ and $(L_1w|L_2w,\ovl{w})$. 
\sn
iii) Since $(K^r|K^h,v)$ is a tame extension, also $(K^rw|K^hw,\ovl{w})$ is a 
tame extension. Indeed, if $(L|K^h,v)$ is a finite subextension, then $p$ does not 
divide $(vL:vK^h)$ and hence also not $(\ovl{w}(Lw):\ovl{w}(K^hw))$, the extension
$Lv|K^hv = (Lw)\ovl{w}|(K^hw)\ovl{w}$ is separable, and $(L|K^h,v)$ is defectless, 
which by ii) implies that $(Lw|K^hw,\ovl{w})$ is defectless. 
\pars
Now as $(K^rw|K^hw,\ovl{w})$ is a tame extension, $K^rw|K^hw$ is separable, and in view of i) we obtain that $K^rw|Kw$ is separable.
\end{proof}

%
%
\subsection{1-units and $p$-th roots in valued fields of mixed characteristic}   \label{Sect1unit}
\mbox{ }\sn
Throughout this section, $(K,v)$ will be a valued field of characteristic zero and
residue characteristic $p>0$, with valuation ring $\cO$ and valuation ideal $\cM$. We
assume that $v$ is extended to the algebraic closure $\tilde{K}$ of $K$.
We will need a few easy observations about the relation of congruences and powers 
of elements.
\begin{lemma}                             \label{congp}
1) If $b_1,\ldots,b_n\in\cO$, then
\begin{equation}                           \label{modp}
(b_1+\ldots+b_n)^p\>\equiv\> b_1^p+\ldots+b_n^p \mod p\cO\>.
\end{equation}
2) Take elements $b_1,\ldots,b_n\in K$ of values $\geq - \frac{vp}{p}$. Then
\[
(b_1+\cdots + b_n)^p \>\equiv\> b_1^p+\cdots + b_n^p \mod \cO\>.
\]
3) Take $\eta\in\tilde{K}$ such that $\eta^p\in\cO$. Then for every $c\in K$ such 
that $v(\eta-c)\geq\frac{vp}{p}$ we have that $\eta^p\equiv c^p \mod p\cO$.
\end{lemma}
\begin{proof}
1): We have:
\begin{equation}                                         \label{b1b2}
(b_1+b_2)^p\>=\>b_1^p+\sum_{i=1}^{p-1} \binom{p}{i} b_1^{p-i} b_2^i +b_2^p\>.
\end{equation}
Since the binomial coefficients under the sum are all divisible by $p$ and since
$b_1,b_2\in\cO$, all summands on the right hand side for $1\leq i\leq p-1$ lie
in $p\cO$, which proves our assertion in the case of $n=2$. The general case follows 
by induction on $n$.

\sn
2): If $vb_1\geq-\frac{vp}{p}$ and $vb_2\geq-\frac{vp}{p}$, then $vb_1^{p-i} b_2^i
\geq -vp$ for $1\leq i\le p-1$, so all summands in the sum on the right hand side of
(\ref{b1b2}) have non-negative value. As for part 1), the
assertion now follows by induction on $n$.
\sn
3): For $c\in K$ with $v(\eta-c)>0$ we have that $vc\geq 0$ and, by part 1):
\[
(\eta-c)^p\>\equiv\>\eta^p-c^p \mod p\cO_{K(\eta)}\>.
\]
If $v(\eta-c)\geq \frac{vp}{p}$, then $v(\eta-c)^p\geq vp$, i.e., $\eta^p-c^p\equiv 
(\eta-c)^p\equiv 0$ modulo $p\cO_{K(\eta)}\cap K=p\cO$.
\end{proof}

\begin{lemma}                                               \label{eta-c}
Take $\eta\in\tilde{K}$ such that $\eta^p\in K$ and $v\eta=0$. Then for $c\in K$ 
such that $v(\eta-c)>0$, $v(\eta-c)<\frac{1}{p-1}vp$ holds if and only if $v(\eta^p
-c^p)<\frac{p}{p-1}vp$, and if this is the case, then $v(\eta^p-c^p)=pv(\eta-c)$. 
If $v(\eta-c)>\frac{1}{p-1}vp$, then $v(\eta^p-c^p)= vp+v(\eta-c)$.
\end{lemma}
\begin{proof}
Take any $c\in K$ such that $0<v(\eta-c)$. Then $vc=v\eta=0$. 
We have: 
\[
\eta^p \>=\> (\eta-c+c)^p\>=\>(\eta-c)^p+\sum_{i=1}^{p-1} \binom{p}{i} 
(\eta-c)^i c^{p-i} \,+\, c^p\>.
\]
Since $vc=0$ and the binomial coefficients under the sum all have value $vp$, the 
unique summand with the smallest value is $p(\eta-c)c^{p-1}$. Therefore,
\begin{equation}                          \label{1u1}
v(\eta^p-c^p)\>\geq\>\min\{v(\eta-c)^p,vp(\eta-c)\} \>=\> 
\min\{pv(\eta-c),vp+v(\eta-c)\}\>,
\end{equation}
with equality holding if $pv(\eta-c)\ne vp+v(\eta-c)$. We observe that
\begin{equation}                          \label{1u2}
v(\eta-c)<\frac{vp}{p-1}\>\Longleftrightarrow\> pv(\eta-c)<vp+v(\eta-c)\>,
\end{equation}
and the same holds for ``$>$'' in place of ``$<$''.
Assume that $v(\eta-c)<\frac{vp}{p-1}$. Then 
\[
v(\eta^p-c^p)\>=\> pv(\eta-c)\><\>\frac{p}{p-1}vp
\]
by (\ref{1u2}) and (\ref{1u1}).
Now assume that $v(\eta-c)\geq \frac{1}{p-1}vp$. Then by (\ref{1u2}), $pv(\eta-c) \geq
vp+v(\eta-c)$, and (\ref{1u1}) yields that
\[
v(\eta^p-c^p)\>\geq\> vp+v(\eta-c) \>\geq\> vp+\frac{1}{p-1}vp \>=\> \frac{p}{p-1}vp\>.
\]
Finally, if $v(\eta-c)>\frac{1}{p-1}vp$, then from (\ref{1u2}) and (\ref{1u1}) 
we conclude: 
\[
v(\eta^p-c^p)\>=\> vp+v(\eta-c)\>.
\]
\end{proof}

A \bfind{1-unit} in $(K,v)$ is an element of the form $u=1+b$ with $b\in\cM$; in 
other words, $u$ is a unit in $\cO$ with residue 1. We will call the value $v(u-1)$ 
the \bfind{level} of the 1-unit $u$. Taking $\eta$ to be a 1-unit $u$ in
Lemma~\ref{eta-c}, we obtain:
\begin{corollary}                                           
Assume that $u$ is a $1$-unit. Then the level of $u$ is smaller than $\frac{1}{p-1}vp$ if 
and only if the level of $u^p$ is smaller than $\frac{p}{p-1}vp$, and if this is the case, 
then $v(u^p-1)=pv(u-1)$.
\end{corollary}

\begin{lemma}                                     \label{apKras}
Take $\eta\in\tilde{K}$ such that $\eta^p\in K$. If there is some $c\in K$ such that
\begin{equation}                                     \label{Kras}
v(\eta-c) \> >\> v\eta + \frac {vp}{p-1} \>,
\end{equation}
then $\eta$ lies in the henselization of $(K,v)$ within $(\tilde{K},v)$.
\end{lemma}
\begin{proof}
If $\eta\in K$, then there is nothing to show, so let us assume that $\eta\notin K$. 
We fix a primitive $p$-th root of unity $\zeta_p\,$.
Every root of $X^p-\eta^p$ is of the form $\eta \zeta_p^i$ with $0\leq i\leq p-1$. For 
$0\leq i\ne j\leq p-1$,
\[
v(\eta \zeta_p^i -\eta \zeta_p^j) \>=\> v\eta+jv\zeta_p+v(\zeta_p^{i-j}-1) \>=\> v\eta+ 
\frac {vp}{p-1}\>,
\]
where the last equality holds since $v\zeta_p=0$ and 
\begin{eqnarray}                                          \label{v_ep}
v(\zeta-1)\>=\>\frac{vp}{p-1}
\end{eqnarray}
for every primitive $p$-th root of unity $\zeta$ (see e.g.\ the proof of \cite[Lemma~2.9]
{[Ku8]}). Hence if (\ref{Kras}) holds, then it follows from Krasner's Lemma that $\eta\in
K(c)^h=K^h$, where $K^h$ denotes the henselization of $(K,v)$ within $(\tilde{K},v)$.
\end{proof}

For our work with $1$-units, we will need a result that is Lemma~2.9 of~\cite{[Ku8]}.
\begin{lemma}                                                  \label{C}
A henselian field of characteristic 0 and residue characteristic $p>0$ contains an element 
$C$ such that $C^{p-1}=-p$ if and only if it contains a primitive $p$-th root $\zeta_p$ of
unity.
\end{lemma}

The element $C$ satisfies:
\begin{equation}                         \label{C^p=-pC}
C^p=-pC \;\; \textrm{ and } \;\; vC=\frac{vp}{p-1}\>.
\end{equation}

The following construction will play an important role in the proof of 
Theorem~\ref{algext}.
Take a 1-unit $\eta\in\tilde{K}$ such that $\eta^p\in K$. Then also $\eta^p$ is a 1-unit.
Assume that $K$ contains an element $C$ as in Lemma~\ref{C}. Consider the substitution 
$X=CY+1$ for the polynomial $X^p-\eta^p$. We then obtain the polynomial
$(CY+1)^p-\eta^p$. Dividing this polynomial by $C^p$ and using the fact that $C^p=-pC$, 
we obtain the polynomial
\begin{equation}                             \label{trpol}
f_\eta(Y)=Y^p+g(Y)-Y-\frac{\eta^p-1}{C^p},
\end{equation}
where
\begin{equation}                              \label{p-part}
g(Y)=\sum_{i=2}^{p-1} \binom{p}{i} C^{i-p}Y^i\>.
\end{equation}

Note that $g(Y)\in\cM_K[Y]$ since $C\in K$ and $vC=\frac{vp}{p-1}$. We see that an element 
$\tilde \eta$ is a root of $X^p-\eta^p$ if and only if the element 
$\frac{\tilde \eta -1}{C}$  is a root of $f_\eta\,$. Thus the roots of
$f_\eta$ are of the form $\frac{\zeta_p^i \eta -1}{C}$ with $0\leq i\leq p-1$.
Set
\begin{equation}                             \label{varthetaeta}
\vartheta_\eta\>:=\>\frac{\eta-1}{C}\>.
\end{equation}
Then $K(\eta)=K(\vartheta_\eta)$, with $f_\eta$ the minimal polynomial of $\vartheta_\eta$ 
over $K$. Modulo $\cM_K[Y]$, the polynomial $f_\eta(Y)$ has the form of an Artin-Schreier
polynomial (see Section~\ref{sectGaldefdegp}). We note that by Lemma~\ref{dist}, 
\[
\dist(\vartheta_\eta,K)\>=\>\dist(\eta,K) - \frac{vp}{p-1}\>.
\]

\pars
The following is an easy consequence of the above.
\begin{lemma}                                                  \label{1-units}
In a henselian field $(K,v)$ of mixed characteristic with residue characteristic $p$
which contains a primitive $p$-th root of unity, every $1$-unit of level strictly 
greater than $\frac{p}{p-1}vp$ is a $p$-th power.
\end{lemma}
\begin{proof}
By Lemma~\ref{C}, $K$ contains an element $C$ as in that lemma. Take a 1-unit $u\in K$ of 
level strictly greater than $\frac{p}{p-1}vp$. Apply the above transformation to the
polynomial $X^p-u$ with $\eta^p=u$. By our assumption on $u$ we have that $\frac{\eta^p-1}
{C^p}\in\cM_K\,$. Hence $f_\eta(Y)$ is equivalent modulo $\cM_K[Y]$ to $Y^p-Y$,
which splits in the henselian field $K$. Therefore, $\vartheta_\eta\in K$, hence 
$\eta\in K$.
\end{proof}

%
%
\subsection{Higher ramification groups}           \label{secthrg}
\mbox{ }\sn
Take a henselian field $(K,v)$. Assume that $L|K$ is a Galois extension, and let
$G=\Gal(L|K)$ denote its Galois group. For ideals $I$ of ${\cal O}_L$ we consider the
(upper series of) \bfind{higher ramification groups}
\begin{equation}
G_I\>:=\>\left\{\sigma\in G\>\left|\;\> \frac{\sigma b -b}{b}\in I
\mbox{ \ for all }b\in L^\times\right.\right\}
\end{equation}
(see \cite{[ZS2]}, \S12). Note that $G_{\cM_L}$ is the ramification group of 
$(L|K,v)$. For every ideal $I$ of ${\cal O}_L\,$, $G_I$ is a
normal subgroup of $G$ (\cite{[ZS2]} (d) on p.79). The function
\begin{equation}                            \label{eq12}
\varphi:\; I\>\mapsto\>G_I
\end{equation}
preserves $\subseteq$, that is, if $I\subseteq J$, then $G_I\subseteq G_J\,$.
As ${\cal O}_L$ is a valuation ring, the set of its ideals is linearly
ordered by inclusion. This shows that also the higher ramification
groups are linearly ordered by inclusion. Note that in general, $\varphi$ will neither 
be injective, nor surjective as a function to the set of normal subgroups of~$G$.

The function
\begin{equation}                           \label{vIS}
v:\; I\>\mapsto\> \Sigma_I\>:=\> \{vb\mid 0\ne b\in I\}
\end{equation}
is an order preserving bijection from the set of all ideals of ${\cal O}_L$ onto the 
set of all final segments of the positive part $(vL)^{> 0}$  of the value 
group $vL$ (including the final segment $\emptyset$). The set of these final segments 
is again linearly ordered by
inclusion, and the function (\ref{vIS}) is order preserving: $J\subseteq I$ holds 
if and only if $\Sigma_J\subseteq\Sigma_I$ holds. The inverse of the above function 
is the order preserving function
\begin{equation}
\Sigma\>\mapsto\> I_\Sigma\>:=\>\{a\in L\mid va\in \Sigma\}\cup\{0\}\>.
\end{equation}

Now the higher ramification groups can be represented as
\[
G_\Sigma\>:=\>G_{I_\Sigma}\>=\>
\left\{\sigma\in G\>\left|\;\> v\,\frac{\sigma b-b}{b}\in \Sigma\cup\{\infty\}
\mbox{ \ for all }b\in L^\times\right.\right\} \>,
\]
where $\Sigma$ runs through all final segments of $(vL)^{> 0}$.

Like the function (\ref{eq12}), also the function $\Sigma\mapsto G_\Sigma$ is in 
general not injective. We call $\Sigma$ a \bfind{ramification jump} if
\[
\Sigma'\subsetneq\Sigma\>\Rightarrow\> G_{\Sigma'}\subsetneq G_\Sigma\>.
\]
If $\Sigma$ is a ramification jump, then $I_\Sigma$ is called a \bfind{ramification
ideal}.

\pars
Given any ramification group $H\subseteq G$, we define
\begin{equation}
\Sigma_-(H)\;:=\;\bigcap_{G_\Sigma=H}^{}\Sigma
\quad\mbox{\ \ and\ \ }\quad
\Sigma_+(H)\;:=\;\bigcup_{G_\Sigma=H}^{}\Sigma\;.
\end{equation}
and note that arbitrary unions and intersections of final segments of $(vL)^{> 0}$
are again final segments of $(vL)^{> 0}$. From its definition it is obvious that 
$\Sigma_-(H)$ is a ramification jump, $G_{\Sigma_-(H)}=H$, and that
\[
I_-(H)\>:=\> I_{\Sigma_-(H)}
\]
is a ramification ideal. It is generated by the set 
\begin{equation}                     \label{genI}
\left\{\frac{\sigma b -b}{b}\>\left|\;\> \sigma\in H\,,\>b\in L^\times\right.
\right\}\>.
\end{equation}

\parm
In this paper we are particularly interested in the case where $(L|K,v)$ is a Galois
extension of prime degree $p$. Then $G=\Gal(L|K)$ is a cyclic group of order $p$ and 
thus has only one proper subgroup, namely $\{\mbox{\rm id}\}$, and this subgroup is 
equal to $G_\Sigma$ for $\Sigma=\emptyset$. If in this case $G$ itself is the
ramification group of the extension, then there must be a unique ramification jump. 
As we will show in the next section, this ramification jump carries important 
information about the extension $(L|K,v)$.

\bn
%
%
%
\section{Defect extensions of prime degree}                     \label{sectdepd}
We will investigate defect extensions $(L|K,v)$ of prime degree $p$. By what we have
already stated in the Introduction, such extensions are immediate unibranched 
extensions; moreover, $p=\chara Kv>0$. By
Theorem~\ref{CutImm}, for every $a\in L\setminus K$ the set $v(a-K)$
is an initial segment of $vK$ without maximal element, and $\dist(a,K)>va$.

\pars
In the following, we distinguish two cases:
\sn
$\bullet$ \ the equal characteristic case where $\chara K=p$,\n
$\bullet$ \ the mixed characteristic case where $\chara K=0$ and $\chara Kv=p$.
\sn
We fix an extension of $v$ from $L$ to the algebraic closure $\tilde{K}$ of $K$.

\pars

\pars
In a first section, we investigate the set $\Sigma_\sigma$ defined in (\ref{Sigsig}) 
for $\sigma$ in the absolute Galois group $\Gal(K):=\Gal(K\sep|K)$.

%
%
\subsection{The set $\Sigma_\sigma$}   \label{sectSigsig}
\mbox{ }\sn
We start with the following easy but helpful observations. The first is obvious.
\begin{lemma}                               \label{vgen}
Let $(K(a)|K,v)$ be any algebraic extension of valued fields. If
$\sigma\in \Gal(K)$ is such that $\sigma a\ne a$, then
\[
\left\{\left. v\>\frac{\sigma(a-c)-(a-c)}{a-c}
\>\right|\> c\in K\right\}\>=\>
\left\{\left. v\>\frac{\sigma a-a}{a-c}
\>\right|\> c\in K\right\}\>=\> -v(a-K)+v(\sigma a-a)\>.
\]
\end{lemma}

\begin{lemma}                               \label{acsaa}
Take a nontrivial immediate unibranched extension $(K(a)|K,v)$. Then the following 
assertions hold.
\sn
1) For each $\sigma\in \Gal(K)$ and $c\in K$,
\[
v(a-c)\><\>v(\sigma a-a)\;.
\]
2) For each $\sigma\in \Gal(K)$ such that $\sigma a\ne a$,
\[
\dist(a,K)\>\leq\> v(\sigma a - a)^-.
\]
\end{lemma}
\begin{proof}
1): Since the extension is immediate and $a\notin K$, the set $v(a-K)$ has no maximal 
element. Thus it suffices to prove that $v(a-c) \leq v(\sigma a-a)$. If this were not 
true, then for some $\sigma\in\Gal(K)$ and $c\in K$, $v(a-c)> v(\sigma a-a)$. But then,
\[
v\sigma(a-c)\>=\>v(\sigma a-c)\>=\> \min\{v(\sigma a-a)\,,\,v(a-c)\}\>=\> v(\sigma a-a)\><\>v(a-c)\;,
\]
which contradicts our assumption that $K(a)|K$ is a unibranched extension, as
$v\sigma$ is also an extension of $v$ from $K$ to $K(a)$.
\sn
2): This is an immediate consequence of part 1).
\end{proof}

\pars
With the help of Lemma~\ref{Kap}, we prove:
\begin{lemma}                               \label{anypol}
Take a defect extension $(K(a)|K,v)$ of prime degree and any $b\in K(a)^\times$. Then 
for all $\sigma \in \Gal(K)$  such that $\sigma a\ne a$ there is some $c\in K$ such 
that
\begin{equation}                            \label{anypole}
v\,\frac{\sigma b -b}{b} \;>\; -v(a-c)\,+\, v(\sigma a-a)\;.
\end{equation}
\end{lemma}
\begin{proof}
As stated already, $(K(a)|K,v)$ is immediate with $[K(a):K]=p=\chara Kv$. 
The element $b\in K(a)^\times$ can be written as $f(a)$ for $f(X)\in K[X]$ of degree
smaller than $p$. By Theorem~\ref{CutImm}, $v(a-K)$ has no maximal element. Hence by
Lemma~\ref{Kap}, we can choose $\gamma\in v(a-K)$ so large that for 
all $c\in K$ with $v(a-c)\geq\gamma$, all values $v\partial_i f(c)$ are
fixed and equal to $v\partial_i f(a)$ whenever $0\leq i<p$, and that (\ref{xf-}) and
(\ref{Kapvfi}) hold. It suffices to restrict our attention to 
those $c\in K$ for which $v(a-c) \geq\gamma$. Then we have that
\begin{equation}                            \label{using}
v\partial_1f(a)(a-c)\>=\>v\partial_1f(c)(a-c)\><\>v\partial_i f(c)(a-c)^i
\>=\>v\partial_i f(a)(a-c)^i
\end{equation}
for all $i>1$. From part 1) of Lemma~\ref{acsaa} we infer that
\[
0\><\>v\left(\frac{\sigma a-a}{a-c}\right)
\><\>v\left(\frac{\sigma a-a}{a-c}\right)^i
\]
for all $i>1$. Using this together with (\ref{using}), we obtain:
\begin{eqnarray*}
v\partial_1 f(a)(\sigma a-a) &=& v\partial_1 f(a)(a-c)\left(\frac{\sigma a-a}{a-c}
\right) \\
&<& v\partial_i f(a)(a-c)^i\left(\frac{\sigma a-a}{a-c}\right)^i\>=\> v\partial_i f(a)
(\sigma a-a)^i\>.
\end{eqnarray*}
It follows that        
\begin{eqnarray*}
v(\sigma f(a) - f(a)) & = & v(f(\sigma a) - f(a))
\>=\>v\left(\sum_{i=1}^{\deg f}\partial_i f(a)(\sigma a-a)^i\right)\\
 & = & v\partial_1 f(a)(\sigma a-a)\>=\>v\partial_1 f(c)\,+\, v(\sigma a-a)\;.
\end{eqnarray*}
Now (\ref{Kapvfi}) shows that
\[
v\partial_1 f(c) + v(a-c) \>=\> v(f(a)-f(c)) \>\geq\> \min\{vf(a),vf(c)\}\>.
\]
The value on the right hand side is fixed, but the value of the left hand side 
increases with $v(a-c)$. Since $v(a-K)$ has no maximal element, we can choose 
$\gamma$ so large that the value on the left hand side is larger
than the one on the right hand side, which can only be the case if $vf(a)=vf(c)$, 
whence $vf(a)< v\partial_1 f(c)+v(a-c)$. Consequently,
\[
v\,\frac{\sigma f(a) -f(a)}{f(a)}\>=\> \partial_1 f(c)\,+\, v(\sigma a-a)
\, -\,vf(a) \>  >\> -v(a-c) \,+\,v(\sigma a-a)\;.
\]
\end{proof}

\begin{theorem}                                        \label{dist_ext_p}
Take a defect extension $\cE= (L|K,v)$ of prime degree. Then for every generator $a\in L$ 
of the extension and every $\sigma \in \Gal(K)$ such that $\sigma a\ne a$,
\begin{equation}                                       \label{ram_gp}
\Sigma_\sigma \>=\> -v(a-K)+v(\sigma a-a)\>,
\end{equation}
and this set is a final segment of $vK^{>0}=\{\alpha\in vK\mid \alpha>0\}$.
\end{theorem}
\begin{proof}
The inclusion ``$\supseteq$'' in (\ref{ram_gp}) follows from Lemma~\ref{vgen}. To
show the reverse inclusion, we use Lemma~\ref{anypol}. 
Since $v(a-K)$ is an initial segment of $vK$, $-v(a-K)$ is a final segment of $vK$. 
Thus we can infer from (\ref{anypole}) that
\[
v\,\frac{\sigma b -b}{b}\>\in \,-v(a-K)\,+\, v(\sigma a-a)\>.
\]
This proves the inclusion ``$\subseteq$''.
\pars
Since $\cE$ is an immediate unibranched extension, taking $c=0$ in part 1) of 
Lemma~\ref{acsaa} yields that 
$v(\sigma b-b)> vb$ for all $b\in L^\times$, showing that $v \frac{\sigma b-b}{b}
\in vL^{>0}=vK^{>0}$. Since $-v(a-K)$ is a final segment of $vK$,
the same holds for $\Sigma_\sigma=-v(a-K)+v(\sigma a-a)$.
\end{proof}

%
%
%
\subsection{Galois defect extensions of prime degree} \label{sectGaldefdegp}
\mbox{ }\sn
A Galois extension of degree $p$ of a field $K$ of characteristic $p>0$ is an
\bfind{Artin-Schreier extension}, that is, generated by an \bfind{Artin-Schreier
generator} $\vartheta$ which is the root of an \bfind{Artin-Schreier
polynomial} $X^p-X-c$ with $c\in K$. A Galois extension of degree $p$ of a field 
$K$ of characteristic 0 which contains all $p$-th roots of unity is a \bfind{Kummer
extension}, that is, generated by a \bfind{Kummer generator} $\eta$ which satisfies 
$\eta^p\in K$. For these facts, see \cite[Chapter VIII, \S8]{[L]}.

If $(L|K,v)$ is a Galois defect extension of degree $p$ of fields of characteristic 
0, then a Kummer generator of $L|K$ can be chosen to be a $1$-unit. Indeed, choose 
any Kummer generator $\eta$. Since $(L|K,v)$ is immediate, we have that $v\eta\in
vK(\eta)=vK$, so there is $c\in K$ such that $vc=-v\eta$. Then $v\eta c=0$, and since 
$\eta cv\in K(\eta)v=Kv$, there is
$d\in K$ such that $dv=(\eta cv)^{-1}$. Then $v(\eta cd)=0$ and $(\eta cd)v=1$. Hence 
$\eta cd$ is a $1$-unit. Furthermore, $K(\eta cd)=K(\eta)$ and $(\eta cd)^p=
\eta^pc^pd^p\in K$. Thus we can replace $\eta$ by $\eta cd$ and assume from the start
that $\eta$ is a $1$-unit. It follows that also $\eta^p\in K$ is a $1$-unit.

Throughout this article, whenever we speak of ``Artin-Schreier extension'' we refer to
fields of positive characteristic, and with ``Kummer extension'' we refer to fields of
characteristic 0.
\begin{theorem}                                        \label{dist_galois_p}
Take a Galois defect extension $\cE=(L|K,v)$ of prime degree with Galois group $G$.
Then $G$ is the ramification group of $\cE$.
The set $\Sigma_\sigma$ does not depend on the choice of
the generator $\sigma$ of $G$. Writing $\Sigma_{\cE}$ for $\Sigma_\sigma\,$, we have 
that $\Sigma_{\cE}$ is a final segment of $vK^{>0}$ and satisfies
\[
\Sigma_{\cE} \>=\> \Sigma_-(G) \>=\> \Sigma_+(\{\mbox{\rm id}\})\>,
\]
showing that $\Sigma_{\cE}$ is the unique ramification jump of the extension $\cE$.
Further, the ramification ideal $I_\cE=I_-(G)$ is equal to the ideal of $\cO_L$ 
generated by the set
\begin{equation}                        \label{genI-G}
\left\{ \left.\frac{\sigma b-b}{b}\, \right| \, b\in L^{\times} \right\}\>,
\end{equation}
for any generator $\sigma$ of $G$.
\pars
If $(L|K,v)$ is an Artin-Schreier defect extension with any Artin-Schreier generator 
$\vartheta$, then
\begin{equation}                            \label{SigmaAS}
\Sigma_{\cE} \>=\> -v(\vartheta-K)\>.
\end{equation}
If $K$ contains a primitive root of unity and $(L|K,v)$ is a Kummer extension with 
Kummer generator $\eta$ of value $0$, then
\begin{equation}                            \label{SigmaKum}
\Sigma_{\cE} \>=\> \frac{vp}{p-1}\,-\,v(\eta-K)\>.
\end{equation}
\end{theorem}
\begin{proof}
It follows from Theorem~\ref{dist_ext_p} that $G$ is the ramification group of $\cE$
and that $\Sigma_{\sigma}$ is a final segment of $vK^{>0}$.

Assume first that $(L|K,v)$ is an Artin-Schreier defect extension with Artin-Schreier
generator $\vartheta$. Then for every generator $\sigma$ of $G$, we have that 
$\sigma \vartheta = \vartheta + i$ for some $i\in \F_p$ and thus, $v(\sigma \vartheta  
- \vartheta)=v i=0$. Hence equation (\ref{ram_gp}) shows that $\Sigma_\sigma$ does not 
depend on the choice of $\sigma$ and that (\ref{SigmaAS}) holds.

Now assume that $K$ contains a primitive root of unity and $(L|K,v)$ is a Kummer
extension with Kummer generator $\eta$ which is a 1-unit. Then $\sigma \eta - \eta
=(\zeta_p-1)\eta$ for some primitive root of unity $\zeta_p\,$, and by equation 
(\ref{v_ep}),
\begin{equation}                           \label{vsa-a}
v(\sigma \eta - \eta)\>=\> v(\zeta_p-1)+v\eta \>=\> \frac{vp}{p-1}\>.
\end{equation}
Hence by equation~(\ref{ram_gp}), $\Sigma_\sigma$ does not depend on the choice of 
$\sigma$, and (\ref{SigmaKum}) holds.

\pars
If $\Sigma\subsetneq \Sigma_\sigma\,$, then $\sigma\notin G_\Sigma$ and hence
$G_\Sigma=\{\mbox{\rm id}\}$. If $\Sigma_\sigma\subseteq \Sigma$, then $\sigma\in
G_\Sigma$ and hence $G_\Sigma=G$. As $\Sigma_{\cE}$ is the intersection of 
all final segments that contain it, 
\[
\Sigma_{\cE}\>=\>\bigcap_{G_\Sigma=G}\Sigma\>=\>\Sigma_-(G) \>.
\]
Since the sets $-v(\vartheta-K)$ in equation (\ref{SigmaAS}) and $-v(\eta-K)$ in 
equation (\ref{SigmaKum}) have no smallest element, the same is true for $\Sigma_{\cE}$.
Therefore, $\Sigma_{\cE}$ is the union of all final segments properly contained in it, 
whence
\[
\Sigma_{\cE}\>=\>\bigcup_{G_\Sigma=\{{\rm id}\}}\Sigma\>=\>
\Sigma_+(\{\mbox{\rm id}\}) \>.
\]
Finally, from Section~\ref{secthrg} we know that $I_-(G)$ is generated by the set
(\ref{genI}). However, as $\Sigma_{\cE}=\Sigma_\sigma$ for every generator $\sigma$ 
of $G$, it is also generated by the set (\ref{genI-G}).
\end{proof}

We define the \bfind{distance of $\cE$} to be the cut
\[
\dist\cE \>:=\> (-\Sigma_\cE)^+
\]
in $\widetilde{vK}$. By applying the distance operator to the right hand sides of
equations (\ref{SigmaAS}) and (\ref{SigmaKum}), we obtain:
\begin{corollary}                           \label{distdist}
If $\cE$ is an Artin-Schreier defect extension, then
\[
\dist\cE\>=\>\dist(\vartheta,K)
\]
for every Artin-Schreier generator $\vartheta$ of $\cE$. Consequently, all Artin-Schreier 
generators of $\cE$ have the same distance.

\pars
If $\cE$ is a Kummer extension, then
\[
\dist\cE\>=\> -\frac{vp}{p-1} \,+\,\dist(\eta,K)\>=\>\dist(\vartheta_\eta,K)
\]
for every Kummer generator $\eta$ of value $0$. Consequently, all Kummer generators of 
$\cE$ of value $0$ have the same distance.
\end{corollary}

\begin{proposition}                                 \label{distAS}
Take a Galois defect extension $\cE=(L|K,v)$ of prime degree $p$.
\sn
1) We have that
\begin{equation}                              \label{Egeq0}
\dist\cE\>\leq\> 0^-\>.
\end{equation}
If $\cE$ is an Artin-Schreier defect extension, then
\begin{equation}                              \label{distleq0}
\dist(\vartheta,K)\>\leq\> 0^-
\end{equation}
for every Artin-Schreier generator $\vartheta$. If $\cE$ is a Kummer defect extension, 
then
\begin{equation}                                  \label{disteta1}
0\><\>\dist(\eta,K) \>\leq\> \left(\frac{vp}{p-1}\right)^-
\end{equation}
for every Kummer generator $\eta$ of value $0$.
\sn
2) The extension $\cE$ has independent defect if and only if
\begin{equation}                                \label{distE=H-}
\dist\cE\>=\>H_\cE^-
\end{equation}
for some proper convex subgroup $H_\cE$ of $vK$. If this is the case, then $vK/H_\cE$ 
has no smallest positive element. In particular, if $vK$ is archimedean,
then $\cE$ has independent defect if and only if $\dist\cE=0^-$.
\sn
3) If $\cE$ is an Artin-Schreier defect extension with Artin-Schreier generator 
$\vartheta$, then it has independent defect if and only if
\begin{equation}                        \label{dist=H-}
\dist(\vartheta,K) \>=\> H_\cE^-
\end{equation}
for some proper convex subgroup $H_\cE$ of $vK$.

A Kummer defect extension of prime degree with Kummer generator 
$\eta$ of value $0$ has independent defect if and only if
\begin{equation}                        \label{dist=+H-}
\dist(\eta,K) \>=\> \frac{vp}{p-1} \,+\, H_\cE^-\>,
\end{equation}
or equivalently,
\begin{equation}                        \label{dist=+H-eq}
\dist(\vartheta_\eta,K) \>=\> H_\cE^-\>,
\end{equation}
for some convex subgroup $H_\cE$ of $vK$. If this holds, then $H_\cE$ does not 
contain $vp$.
\end{proposition}
\begin{proof}
1): Inequality (\ref{Egeq0}) follows from Theorem~\ref{dist_ext_p} together with the 
definition of $\Sigma_\cE$ in Theorem~\ref{dist_galois_p}. From inequality (\ref{Egeq0}) 
we obtain inequality (\ref{distleq0}) and the second inequality in (\ref{disteta1}) by an
application of Corollary~\ref{distdist}. The first inequality
in (\ref{disteta1}) follows from Theorem~\ref{CutImm} since $v\eta=0$.
\sn
2): By definition, $\dist\cE=(-\Sigma_\cE)^+$. Applying Lemma~\ref{S^+=H^-} to the 
initial segment $\Sigma=-\Sigma_\cE$, we see that $(-\Sigma_\cE)^+=H_\cE^-$ holds if and 
only if $\,\Sigma_\cE=\{\alpha\in vK\mid \alpha >H_\cE\}$ and $vK/H_\cE$ has no smallest
positive element. This is the definition for $\cE$ to have independent defect.

The final assertion of part 2) follows from the facts that the only proper convex 
subgroup in an archimedean ordered abelian group is $\{0\}$ and that $\dist\cE$ cannot 
have a largest element because the extension $\cE$ is immediate.
\sn
3): This follows from part 2) together with Corollary~\ref{distdist}. In the case 
of a Kummer extension we have that $\dist(\eta,K)>v\eta=0$, so $H_\cE$ cannot contain 
$vp$. 
\end{proof}

\begin{proposition}                                    \label{idfKK^r}
Take a Galois defect extension $\cE=(L|K,v)$ of prime degree $p$ with an 
Artin-Schreier or Kummer generator $a$. Further, choose any extension of $v$ from $K(a)$ 
to $\tilde K$, take $(K^r,v)$ to be the absolute ramification field of $(K,v)$, and $N$ 
to be an intermediate field of $K^r|K$. Then also $\cE_N:=
(L.N|N,v)$ is a Galois defect extension of degree $p$, 
\[
\dist\cE_N\>=\>\dist\cE\>,
\]
and $\cE_N$ has independent defect if and only $\cE$ has. Further, if $(N,v)$ is an
independent defect field, then so is $(K,v)$.
\end{proposition}
\begin{proof}
We may assume that $a$ is a generator of $\cE$ as in Theorem~\ref{dist_galois_p}. By 
equation (\ref{distKN}) of Proposition \ref{K(a)K^r(a)}, also $\cE_N$ is a Galois
defect extension of prime degree $p$, and $\dist(a,N)=\dist(a,K)$. In view of 
Corollary~\ref{distdist}, we obtain that $\dist\cE_N=\dist\cE$. From this,
the third assertion follows by part 2) of Proposition~\ref{distAS}. 
\pars
In order to prove the final assertion, assume that $(N,v)$ is an independent defect 
field. Take a $p$-th root of unity $\zeta_p$. Then by definition, also $N(\zeta_p)
\subseteq K^r=K(\zeta_p)^r$ is an independent defect field with respect to $v$. 
Take any Galois defect extension of degree $p$ of $K(\zeta_p)$ with a generator $a$ as 
above. Then $(N(\zeta_p)(a)|N(\zeta_p),v)$ has
independent defect, and by what we have proved already, the same is true for 
$(K(\zeta_p)(a)|K(\zeta_p),v)$. This shows that $(K,v)$ is an independent defect field.
\end{proof}

\mn
{\it Proof of Theorem~\ref{eqindep}.} 
We have shown in Theorem~\ref{dist_galois_p} that $\Sigma_{\cE}$ is the unique 
ramification jump of $\cE$, and it follows that $I_\cE$ is the unique ramification ideal 
of $\cE$. Thus the equivalence of assertions a) and b) follows from the definition of
independent defect in cases where the condition that $vK/H_\cE$ has no smallest positive
element always holds. It is straightforward to show that this happens when $(K,v)$ 
satisfies (DRvg). The equivalence of assertions b) 
and c) in Theorem~\ref{eqindep} is always valid because
$vL=vK$ and an ideal $I_\Sigma$ of $\cO_L$ is prime if and only if $\Sigma=\{\alpha\in 
vL\mid \alpha >H\}$ for some proper convex subgroup $H$ of $vL$.

The remaining assertions follow from basic facts of valuation theory.
\qed

\parm
For Artin-Schreier defect extensions, a different definition was given for dependent 
and independent defect in \cite{[Ku6]}. We will show in the next section that our 
new definition is consistent with the previous one.

%
%
%
\subsection{Artin-Schreier defect extensions}              \label{sectASde}
\mbox{ }\sn
In this section, we consider the case of a valued field $(K,v)$ of positive
characteristic $p$ and an Artin-Schreier defect extension $(L|K,v)$ with 
Artin-Schreier generator $\vartheta$, that is,
$\vartheta^p-\vartheta\in K$. The following definition was introduced in~\cite{[Ku6]}:
if there is an immediate purely inseparable extension $(K(\eta)|K,v)$ of degree $p$ 
such that
\begin{equation}                \label{depdef}
\vartheta \>\sim_K\> \eta\>,
\end{equation}
then we say that the Artin-Schreier defect extension has \bfind{dependent defect};
otherwise it has \bfind{independent defect}. Note that (\ref{depdef}) implies that 
$\dist(\eta,K)<\infty$, that is, $\eta$ does not lie in the completion of $(K,v)$, 
since otherwise it would follow that $\vartheta =\eta$.

The above definition does not depend on the Artin-Schreier generator of the extension 
$L|K$. Indeed, by \cite[Lemma 2.26]{[Ku6]}, $\vartheta'\in L$ is another Artin-Schreier 
generator of $L|K$ if and only if  $\vartheta' =i\vartheta +c$ for some $i\in\F_p^\times$ 
and $c\in K$. If we set $\eta'=i\eta+c$, then $K(\eta)=K(\eta')$ and $v(\vartheta'-
\eta')=v(i(\vartheta-\eta))=v(\vartheta-\eta)>\dist(\vartheta,K)$, that is, 
$\vartheta' \sim_K \eta'$. 

\begin{proposition}                              \label{ASindep}
An Artin-Schreier defect extension has dependent defect in the sense as defined in
\cite{[Ku6]} if and only if it has dependent defect in the sense as defined in the
introduction.
\end{proposition}
\begin{proof}
Take an Artin-Schreier defect extension $\cE=(L|K,v)$ with Artin-Schreier generator 
$\vartheta$, and write $p=\chara K$. 

First assume that $\vartheta\sim_K\eta$ holds for some element $\eta$ such that 
$\eta^p\in K$. Then $v(\vartheta-\eta)>
\dist(\vartheta,K)$. On the other hand, $pv(\vartheta-\eta)=v(\vartheta^p-\eta^p)=
v(\vartheta + \vartheta^p-\vartheta - \eta^p)\in v(\vartheta-K)$. Suppose that 
(\ref{dist=H-}) holds for some proper convex subgroup $H_\cE$ of $vK$. Then 
$\alpha\leq v(\vartheta-\eta)$ for some $\alpha\in H_\cE$, but 
$p\alpha\leq pv(\vartheta-\eta)< H_\cE$, a contradiction. Now Proposition~\ref{distAS} 
shows that $(L|K,v)$ has dependent defect in the sense as defined in the introduction.

\pars
To prove the converse, assume that $(L|K,v)$ has dependent defect in the sense as 
defined in the introduction. This means that either there is a convex subgroup $H$ of 
$vK$ such that $v(\vartheta-K)=vK^{\leq 0}\setminus H$ but $vK/H$ has a smallest 
positive element, or that there is no such convex subgroup at all. In the former case, 
take $\gamma\in vK^{\leq 0}$ such that $-\gamma+H$ is that smallest positive element.
Then $0>\gamma\notin H$, hence $\gamma\in v(\vartheta-K)$. It follows that for every 
$\beta\in v(\vartheta-K)$, $\beta+H\leq \gamma+H<0$ and therefore, $p(\beta+H)<\gamma+H$.
This implies that $p\beta<\gamma$, showing that $pv(\vartheta-K)<\gamma$. Choosing
$c\in K$ such that $\gamma=v(\vartheta-c)$, we obtain:
\begin{equation}                               \label{prep4.5c)}
pv(\vartheta-c-K) \>=\>pv(\vartheta-K)\><\>v(\vartheta-c)\>.
\end{equation}

On the other hand, if there is no such convex subgroup at all, then the set $\{\gamma\in
vK\mid v(\vartheta-K)<\beta\leq 0\}$ is not closed under addition; more specifically, 
there is some $\beta$ in this set and some $c\in K$ such that $p\beta\leq v(\vartheta
-c)$. As $pv(\vartheta-K)<p\beta$, we again obtain that (\ref{prep4.5c)}) holds.

Set $a:=(\vartheta-c)^p-(\vartheta-c)
\in K$ so that the Artin-Schreier generator $\vartheta-c$ becomes a root of the 
Artin-Schreier polynomial $X^p-X-a$. Then by \cite[Theorem 4.5 (c)]{[Ku6]},
the root $\eta$ of the polynomial $X^p-a$ generates an immediate extension which does 
not lie in the completion of $(K,v)$, and $\vartheta-c\sim_K\eta$ holds.
\end{proof}

\pars
The name ``dependent defect'' was chosen in \cite{[Ku6]} because the existence of 
Artin-Schreier defect extensions with a defect of $(K,v)$ that is dependent according 
to the definition given in \cite{[Ku6]} depends on the existence 
of purely inseparable defect extensions of degree $p$ that do not lie in the completion; 
\cite[Proposition 4.3]{[Ku6]} shows how the former are constructed from the latter. If 
$(K,v)$ admits any purely inseparable defect extension not contained in its completion, 
then it also admits one of degree $p$. This is proved in the beginning of Section 4.3 of
\cite{[Ku6]}. This then leads to an Artin-Schreier defect extension with dependent defect.

The reverse construction is given in the foregoing proof. Hence if $(K,v)$ has 
an Artin-Schreier defect extension with dependent defect in the sense as defined in 
the introduction, then it admits an immediate purely inseparable extension of degree 
$p$ that does not lie in the completion of $(K,v)$. 
This proves part 2) of Proposition~\ref{idf}.

\bn
%
%
\section{Semitame, deeply ramified and rdr fields}                     \label{sectdr}
{\bf Throughout this section, we will consider a valued field $(K,v)$ of residue 
characteristic $p>0$.} All statements we will prove are trivial for valued fields 
of residue characteristic $0$.

When we deal with valued fields $(K,v)$ of mixed characteristic with residue
characteristic $p$, we will write 
$v=v_0\circ v_p\circ\ovl{v}$ as in the paragraph before Proposition~\ref{rdrrk1}, 
set $\erf(K,v):=Kv_0 v_p$ and denote by $(vK)_{vp}$ the smallest convex subgroup of 
$vK$ that contains $vp$. Further, $\frac{1}{p^\infty}\Z vp$ will denote the 
$p$-divisible hull of the subgroup $\Z vp$ of $vK$ generated by $vp$.
If $K$ has positive characteristic $p$, then we set 
$\erf(K,v):=Kv$ and $(vK)_{vp}=Kv$.

%
%
\subsection{Some basic results}                     
\mbox{ }\sn
To start with, we state a few simple observations.
\begin{lemma}                                     \label{OK/pOK}
1) If $\chara K=p>0$, then
\begin{equation}                                  \label{OK/pOKeq}
\cO_K/p\cO_K \ni x\mapsto x^p\in \cO_K/p\cO_K
\end{equation}
is surjective if and only if $K$ is perfect; in particular, (DRvr) holds if and only 
if $\hat K$ is perfect.
\mn
2) If (\ref{OK/pOKeq}) is surjective, then (DRvr) holds.
\mn
3) If $\chara K=0$, then the following assertions are equivalent:
\pars
a) (\ref{OK/pOKeq}) is surjective,
\pars
b) for every $a\in\cO_K$ there is $c\in\cO_K$ such that $a\equiv c^p\mod p\cO_K\,$,
\pars
c) for every $\hat a\in\cO_{\hat K}$ there is $c\in\cO_K$ such that $\hat a\equiv 
c^p\mod p\cO_{K(\hat a)}\,$,
\pars
d) (DRvr) holds.
\mn
4) If $(K,v)$ satisfies (DRvr), then so does every extension of $(K,v)$ within its completion.
\end{lemma}
\begin{proof}
1): From $\chara K=p>0$ it follows that $p\cO_K=\{0\}$, hence the surjectivity of the
homomorphism in (\ref{homOpO}) means that every element in $\cO_K$ is a $p$-th power. 
Hence the same is true for every element in $K$, i.e., $K$ is perfect. Replacing $K$ by 
$\hat K$ in (\ref{OK/pOKeq}), we thus obtain
that $\hat K$ is perfect.

\sn
2): Assume first that $\chara K=p>0$. Then by part 1) the surjectivity of 
(\ref{OK/pOKeq}) implies that $K$ is perfect. Since the completion of a perfect field 
is again perfect, it follows that $\hat K$ is perfect. Hence again
by part 1), (DRvr) holds.

Now assume that $\chara K=0$. Take $\hat a\in\cO_{\hat K}\,$. Then there exists 
$a\in K$ such that $\hat a\equiv a \mod p\cO_{\hat K}\,$. By assumption, there is 
some $c\in\cO_K$ such that $a\equiv c^p\mod p\cO_K\,$. It follows
that $\hat a\equiv a\equiv c^p\mod p\cO_{\hat K}\,$, showing that (DRvr) also 
holds in this case.

\sn
3): Assume that $\chara K=0$. The proof of the equivalence of a) and b) is
straightforward. Trivially, c) implies b), and part 2) of our lemma shows that a) 
implies d). To show that d) implies c), take $\hat a\in\cO_{\hat K}\,$. Then by 
(DRvr), using the equivalence of a) and b) with $\hat K$ in place of $K$, there is 
$\hat c\in \cO_{\hat K}$ such that $\hat a\equiv {\hat c}^p\mod p\cO_{\hat K}\,$. 
We take $c\in\cO_K$ such that $c\equiv {\hat c}\mod p\cO_{\hat K}$.
Then $\hat a\equiv {\hat c}^p\equiv c^p\mod p\cO_{\hat K}\,$, whence 
$\hat a\equiv c^p\mod p\cO_{K(\hat a)}\,$.

\sn
4): Take $(L|K,v)$ to be a subextension of $(\hat K |K,v)$. Then $\hat L =\hat K$, 
and in the case of $\chara K=p>0$ our assertion follows from part 1).

Now assume that $(K,v)$ is of mixed characteristic and satisfies (DRvr). Then by the
implication d)$\Rightarrow$c) of part 3), for every $\hat a\in\cO_{\hat K}=
\cO_{\hat L}$ there is $c\in\cO_K\subseteq \cO_L$ such that $\hat
a\equiv c^p\mod p\cO_{K(\hat a)}\,$. Hence (\ref{OK/pOKeq}) is surjective in $(L,v)$, 
and the implication a)$\Rightarrow$d) of part 3) shows that $(L,v)$ satisfies (DRvr).
\end{proof}

\begin{lemma}                                     \label{basprop1}
If $(K,v)$ satisfies (DRvr), then the following assertions hold:
\sn
1) The residue fields $Kv$ and $\erf(K,v)$ are perfect.
\sn
2) If $\chara K=p>0$, then $vK$ is $p$-divisible and $(K,v)$ is a semitame field.
\end{lemma}
\begin{proof}
To prove part 1), take any $a\in\cO$. By assumption, there is $\hat c\in \cO_{\hat K}$
such that $a\equiv {\hat c}^p \mod p\cO_{\hat K}\,$. From this we obtain that 
$av={\hat c}^p v= ({\hat c}v)^p\in \hat K v=Kv$. Hence $Kv$ is perfect. If $(K,v)$ is
of mixed characteristic, then the same holds with $v_0\circ v_p$ in place of $v$, 
which shows that $\erf(K,v)$ is perfect.

To prove part 2), assume that $\chara K=p>0$. Then by part 1) of Lemma~\ref{OK/pOK},
(DRvr) implies that $\hat K$ is perfect, so $vK=v\hat K$ is $p$-divisible and (DRst)
holds, showing that $(K,v)$ is a semitame field.
\end{proof}

Take any $d\in \cM_K\,$. If for every $a\in\cO_K$ there is $c\in\cO_K$ such that 
$a\equiv c^p\mod d\cO_K\,$, we will say that the function
\begin{equation}                                        \label{OK/dOKeq}
\cO_K \ni x\mapsto x^p\in \cO_K
\end{equation}
is \bfind{surjective modulo $d\cO_K\,$}. This implies that the function
\begin{equation}                                        \label{OK/dOKeqtimes}
\cO_K^\times\ni x\mapsto x^p\in \cO_K^\times
\end{equation}
is \bfind{surjective modulo $d\cO_K\,$} (with the obvious modification of the above
definition). 

\begin{lemma}                                     \label{basprop2}
For a valued field $(K,v)$ of mixed characteristic, the following assertions hold:
\sn
1) If $(K,v)$ is an rdr field, then $(vK)_{vp}$ is $p$-divisible; in particular,
$vK$ contains $\frac{1}{p^\infty}\Z vp$. 
If in addition $(vK)_{vp}=vK$, then $(K,v)$ is a semitame field.
\mn
2) Assume that for $d\in \cM_K$ the function (\ref{OK/dOKeqtimes}) is
surjective modulo $d\cO_K\,$. Then for every $a\in K$ with $p$-divisible value 
$va$ there is $c\in K$ such that
\begin{equation}                              \label{v(a-cp)}
v(a-c^p) \>\geq\> va+vd\>.
\end{equation}
If in addition $vd\in (vK)_{vp}$ and $(vK)_{vp}$ is $p$-divisible, then the function 
(\ref{OK/dOKeq}) is surjective modulo $d\cO_K\,$. 
\end{lemma}
\begin{proof}
1): First, let us show that every $\alpha\in vK$ with $0\leq \alpha<vp$ is divisible 
by $p$. Take $a\in \cO$ such that $va=\alpha$. From (DRvr) we obtain that there is 
$\hat c\in \cO_{\hat K}$ such that $a\equiv {\hat c}^p \mod p\cO_{\hat K}\,$. Since
$va<vp$, this yields that $va=v{\hat c}^p=pv{\hat c}$, showing that
$\alpha=va$ is divisible by $p$ in $v\hat K=vK$.

By assumption, $vp$ is not the smallest positive element in $vK$, hence there is 
$\alpha\in vK$ such that $0<\alpha<vp$, and we know that $\alpha$ is divisible by $p$. 
We may assume that $2\alpha\geq vp$ since otherwise we replace $\alpha$ by $vp-\alpha$.
In this way we make sure that $(vK)_{vp}$ is equal to the smallest convex
subgroup containing $\alpha$. This implies that for every $\beta\in (vK)_{vp}$ there 
is some $n\in\Z$ such that $0\leq\beta-n\alpha<vp$. Then by what we have already shown, 
$\beta-n\alpha$ is divisible by $p$. Since also
$\alpha$ is divisible by $p$, the same is consequently true for $\beta$.

If in addition $(vK)_{vp}=vK$, then $vK$ is $p$-divisible, and since (DRvr) holds by
assumption, $(K,v)$ is a semitame field.

\sn
2): Take $a\in K$ with $p$-divisible value. Then there is $b\in K$ such that 
$pvb=va$. Hence $vb^{-p}a=0$ and by assumption, there is $c_0\in K$ 
such that $v(b^{-p}a-c_0^p)\geq vd$, whence
\[
v(a-(bc_0)^p) \>=\> pvb+v(b^{-p}a-c_0^p) \>\geq\> va + vd\>.
\]
With $c:=bc_0\,$, this yields (\ref{v(a-cp)}).

Now assume in addition that $vd\in (vK)_{vp}$ and $(vK)_{vp}$ is $p$-divisible, and take 
$a\in \cO_K\,$. If $va>(vK)_{vp}\,$, then $a\equiv 0^p \mod d\cO_K\,$. If $va\in 
(vK)_{vp}\,$, then $va$ is $p$-divisible and by what we have already shown there is 
$c\in K$ such that $a\equiv c^p \mod d\cO_K\,$. This proves that (\ref{OK/dOKeq}) is
surjective modulo $d\cO_K\,$.
\end{proof}

\begin{proposition}                                   \label{dpd'1}
Take a valued field $(K,v)$ of mixed characteristic such that $(vK)_{vp}$ 
is $p$-divisible. Further, take $d\in \cM_K$ such that $vd<vp$ and $nvd\geq vp$ for 
some $n\in\N$. Then the following assertions are equivalent:
\sn
a) the function (\ref{OK/pOKeq}) is surjective, so $(K,v)$ is an rdr field,
\sn
b) the function (\ref{OK/dOKeq}) is surjective modulo $d\cO_K\,$,
\sn
c) the function (\ref{OK/dOKeqtimes}) is surjective modulo $d\cO_K\,$.
\end{proposition}
\begin{proof}
Since $vd<vp$, we have that $p\cO_K\subset d\cO_K\,$. Hence the proof of
implication a)$\Rightarrow$c) is straightforward. Implication c)$\Rightarrow$b) is the
content of part 2) of Lemma~\ref{basprop2}.
\sn
b)$\Rightarrow$a): Assume that assertion b) holds, and take any $a\in\cO_K\,$. By 
part 2) of Lemma~\ref{basprop2}, there is $c_1\in K$ such that $v(a-c_1^p)\geq vd$.
Now we proceed by induction: having chosen $c_k$ such that 
\[
v(a-c_1^p-\ldots -c_k^p)\>\geq\> kvd\>,
\]
we can employ part 2) of Lemma~\ref{basprop2} again to find $c_{k+1}\in K$ such that
\[
v(a-c_1^p-\ldots -c_k^p-c_{k+1}^p)\geq kvd+vd\>.
\] 
After $n$ many steps we have:
\[
v(a-c_1^p-\ldots -c_n^p)\>\geq\> nvd\geq vp\>.
\]
Using part 1) of Lemma~\ref{congp}, we obtain: 
\[
a \>\equiv\> c_1^p+\ldots +c_n^p \>\equiv\> (c_1^p+\ldots +c_n)^p\>\mod p\cO_K\>.
\]
This proves that the function (\ref{OK/pOKeq}) is surjective. By part 2) of
Lemma~\ref{OK/pOK}, (DRvr) holds. By assumption, $(vK)_{vp}$ 
is $p$-divisible, hence also (DRvp) holds. This proves that $(K,v)$ is an rdr field.
\end{proof}

\parm
\begin{lemma}                                                        \label{maxnots}
Assume that $(K,v)$ is of mixed characteristic with $(vK)_{vp}$ $p$-divisible and $Kv$ 
perfect, and take $\eta\in \tilde{K}$ such that $\eta^p\in\cO_K\,$. Then either 
$v(\eta-K)$ does not admit a maximal element, or its maximal element is not smaller 
than $\frac{vp}{p}$.
\end{lemma}
\begin{proof}
Take $c\in K$ such that $0\leq v(\eta-c)<\frac{vp}{p}$. Then by use of 
(\ref{modp}) it follows that $v(\eta^p-c^p)=v(\eta-c)^p<vp$. Since $(vK)_{vp}$ is 
$p$-divisible, there is some $d_1\in K$ such that 
$vd_1^p(\eta^p-c^p)=0$, and since $Kv$ is perfect, there is some $d_2\in K$ such
that $(d_2^pd_1^p(\eta^p-c^p))v=1$. With $d=(d_1d_2)^{-1}$ it follows that
$v(d^{-p}(\eta^p-c^p)-1)>0$, whence $v(\eta^p-c^p-d^p)>v(\eta^p-c^p)$. Again by
(\ref{modp}), we obtain that $(\eta-c-d)^p\equiv \eta^p-c^p-d^p \mod p\cO$, and it
follows that $v(\eta-c-d)>v(\eta-c)$.
\end{proof}

%
%
\subsection{Proof of Theorem~\ref{connprop}}                    
\mbox{ }\sn
1): Assume that $(K,v)$ is nontrivially valued. The implication tame field
$\Rightarrow$ separably tame field is obvious, and so is the implication semitame 
field $\Rightarrow$ deeply ramified field. To prove the implication deeply ramified 
field $\Rightarrow$ rdr field, we first observe that if $\chara K=p>0$, then 
$vp=\infty$ which is not the smallest positive element of $vK$. If $\chara K=0$, then 
$vp$ is not the smallest positive element of $vK$ since otherwise, if $\Gamma$ is the 
largest convex subgroup of $vK$ not containing $vp$, then
$(vK)_{vp}/\Gamma\isom\Z$ in contradiction to (DRvg).

Now assume that $(K,v)$ is a separably tame field. If $\chara K>0$, then
by \cite[Corollary 3.12]{[K7]}, $(K,v)$ is dense in its perfect hull. Then the 
completion of the perfect hull is also the completion of $(K,v)$. Since the 
completion of a perfect valued field is again perfect, we obtain that
the completion of $(K,v)$ is perfect. Now part 1) of Lemma~\ref{OK/pOK}
shows that $(K,v)$ is a semitame field. 

Assume that $\chara K=0$. Then the separably tame field $(K,v)$ is a tame field.
By \cite[Lemma 3.1 and Theorem 3.2]{[K7]}, $(K,v)$ is defectless, $vK$ 
is $p$-divisible and $Kv$ is
perfect. Take any $b\in K$ that is not a $p$-th power, and take $\eta\in\tilde{K}$ 
with $\eta^p=b$. The unibranched extension $(K(\eta)|K,v)$ is defectless, hence by 
Lemma~\ref{imm_deg_p}, $v(\eta-K)$ has a maximal element.
By Lemma~\ref{maxnots}, this maximal element is not smaller than $\frac{vp}{p}$. 
Now part 3) of Lemma~\ref{congp} shows the existence of $c\in K$ such
that $b\equiv c^p \mod p\cO_K\,$. This proves that $(K,v)$ is a semitame field.

\sn
2): Assume that $(K,v)$ is an rdr field of rank 1 and mixed characteristic. Since the 
rank is 1, we have that $(vK)_{vp}=vK$. Hence by part 1) of Lemma~\ref{basprop2}, 
$(K,v)$ is a semitame field. This together with part 1) of our theorem shows the 
required equivalence in the case of mixed characteristic. For the case of equal
characteristic, it will be shown in the proof of part 3).

\sn
3): Assume that $(K,v)$ is a nontrivially valued field of characteristic $p>0$.
\sn
The implications a)$\Rightarrow$b)$\Rightarrow$c) have already been shown in part 1).
\sn
c)$\Rightarrow$d): This holds by definition.
\sn
d)$\Rightarrow$e): This holds by part 1) of Lemma~\ref{OK/pOK}.
\sn
e)$\Rightarrow$f): If the completion of $(K,v)$ is perfect, then it contains the perfect 
hull of $K$; since $(K,v)$ is dense in its completion, it is then also dense in its 
perfect hull.
\sn
f)$\Rightarrow$g): If $(K,v)$ is dense in its perfect hull, then in particular it is dense 
in $K^{1/p}=\{a^{1/p}\mid a\in K\}$. Since $x\mapsto x^p$ is an isomorphism which preserves
valuation divisibility, the latter holds if and only if $(K^p,v)$ is dense in $(K,v)$.
\sn
g)$\Rightarrow$f): Assume that $(K^p,v)$ is dense in $(K,v)$. Since for each $i\in\N$,
$x\mapsto x^{p^i}$ is an isomorphism which preserves valuation divisibility, it follows that 
$(K^{1/p^{i-1}},v)$ is dense in $(K^{1/p^i},v)$. By transitivity
of density we obtain that $(K,v)$ is dense in $(K^{1/p^i},v)$ for each $i\in\N$, and 
hence also in its perfect hull.
\sn
f)$\Rightarrow$e): This implication was already shown in the proof of part 1) of our theorem.
\sn
e)$\Rightarrow$a): Assume that $\hat K$ is perfect. The extension $(\hat K|K,v)$ is
immediate, so $vK=v\hat K$, which is $p$-divisible. Hence (DRst) holds. By part 1) of
Lemma~\ref{OK/pOK}, also (DRvr) holds.
\sn
4): The assertion follows from the implication f)$\Rightarrow$a) of part 3) as a perfect 
field is equal to its perfect hull.
\qed

%
%
\subsection{Proof of Propositions~\ref{rdrrk1} and~\ref{wcow}}        
\mbox{ }\sn
For the proof of Propositions~\ref{rdrrk1} and~\ref{wcow}, we will need some 
preparation.
\begin{lemma}                        \label{rdrcoars}
Assume that $(K,v)$ is of mixed characteristic, and set $w:=v_p\circ\ovl{v}$. Then 
$(K,v)$ is an rdr field if and only if $(Kv_0,w)$ is an rdr field.
\end{lemma}
\begin{proof}
First assume that $(K,v)$ is an rdr field. Then $vp$ is not the smallest positive 
element in $vK$, which implies that $w p$ is not the smallest element in $w (Kv_0)$. 
Take any $b\in\cO_{Kv_0}\,$. Then choose $a\in \cO_K$ such that $av_0=b$. Since 
$(K,v)$ is an rdr field, there is some $c\in \cO_K$ such that $a-c^p\in p\cO_K$. It
follows that $cv_0\in \cO_{Kv_0}$ with $b-(cv_0)^p= (a-c^p)v_0\in p\cO_{Kv_0}\,$,
showing that $(Kv_0,w)$ satisfies (DRvr) by part 3) of Lemma~\ref{OK/pOK}. Hence 
$(Kv_0,w)$ is an rdr field.

\pars
Now assume that $(Kv_0,w)$ is an rdr field. Then $w p$ is not the smallest element in
$w(Kv_0)$, which implies that $vp$ is not the smallest positive element in $vK$. Take any 
$a\in \cO_K\,$. Then $av_0\in\cO_{Kv_0}\,$ and there is some $d\in\cO_{Kv_0}$ such that 
$av_0- d^p\in p\cO_{Kv_0}\,$. Choose $c\in\cO_K$ such that $cv_0=d$. It follows that 
$a-c^p\in p\cO_K\,$. Again using part 3) of Lemma~\ref{OK/pOK}, we conclude that $(K,v)$ 
is an rdr field.
\end{proof}

\mn
{\it Proof of Proposition~\ref{rdrrk1}.}
\n
The equivalence of assertions a) and b) is proved in Lemma~\ref{rdrcoars}. 
To prove the equivalence of assertions a) and c), we may assume that $v_0$ is trivial, 
that is, $Kv_0=K$, $v=v_p\circ\ovl v$ and $vK=(vK)_{vp}\,$. The assertion is trivial if 
$\ovl v$ is trivial, so we assume that it is not. This implies that $vp$ is not the 
smallest positive element in $vK$.

Let us first assume that $(K,v)$ is an rdr field. Then $\frac{vp}{p}\in vK$ by part 1) 
of Lemma~\ref{basprop2}, so $\frac{v_p p}{p}\in v_pK$, showing that $v_p p$ is not the
smallest positive element in $v_p K$. It remains to show that $(K,v_p)$ satisfies 
(DRvr); by part 3) of Lemma~\ref{OK/pOK} it suffices to prove that (\ref{OK/pOKeq})
is surjective in $(K,v_p)$. Take any $a\in \cO_{v_p}\,$. Since $(K,v)$ is an rdr field, 
by part 2) of Lemma~\ref{basprop2} there is
$c\in K$ such that $v(a-c^p)\geq va+vp$, whence $v_p(a-c^p)\geq v_p a+v_p p\geq v_p p$.

Now assume that $(K,v_p)$ is a deeply ramified field, hence an rdr field. As $\ovl{v}$ 
is not trivial, we know already that (DRvp) holds in $(K,v)$,
so it remains to show that (\ref{OK/pOKeq}) holds. Since $(K,v_p)$ is an rdr field, 
for every $a\in\cO_v\subseteq \cO_{v_p}$ there is some $c\in K$ such that 
$v_p(a-c^p)\geq v_p p$, whence $v(a-c^p)>\frac{vp}{p}$. Choosing $d\in K$ with
$vd=\frac{vp}{p}$ and applying Proposition~\ref{dpd'1}, we conclude that 
$(K,v)$ is an rdr field. 

\pars
We turn to the equivalence of assertions a) and d). The implication a)$\Rightarrow$d)
follows from part 1) of Lemma~\ref{basprop2}. Conversely, if $vK$ is roughly 
$p$-divisible, then $vp$ itself is divisible by $p$, so (DRvp) holds. 
\qed

\mn
{\it Proof of Proposition~\ref{wcow}.}\n
Take an arbitrary valued field $(K,v)$ and assume that $v=w\circ\ovl{w}$ with $w$ and 
$\ovl{w}$ nontrivial.
Assume first that $\chara K>0$. Then by part 3) of Theorem~\ref{connprop}, the 
properties ``semitame'', ``deeply ramified'' and ``rdr'' are equivalent, so we have 
to prove the assertions of the proposition only for ``rdr''.

As $w$ is nontrivial and a coarsening of $v$, the topologies generated by $v$ and 
$w$ are equal, and it follows that $(K,v)$ is dense in its perfect hull if and only 
if the same holds for $(K,w)$. By the equivalence of assertions c) and f) in part 3) 
of Theorem~\ref{connprop}, it follows that $(K,v)$ is an rdr field if and only if 
$(K,w)$ is an rdr field. If the latter is the case, then from  Lemma~\ref{basprop1} 
we see that $Kw$ is perfect, and as it is of positive characteristic like $K$, we 
obtain from part 3) of Theorem~\ref{connprop} that $(Kw,\ovl{w})$ is also an rdr field.

Now we assume that $\chara K=0$ and prove the assertions for the property ``rdr''. 
First we discuss the case where $\chara Kw>0$ and  write $w$ in the same way as we 
do for $v$: $w=w_0\circ w_p\circ\ovl{w}$. Then $v_0=w_0$,
$v_p=w_p$, and $\ovl{w}$ is a (possibly trivial) coarsening of $\ovl{v}$. Hence it
follows from Proposition~\ref{rdrrk1} that $(K,v)$ is an rdr field if and only if 
$(K,w)$ is an rdr field. If the latter is the case, then because of $\chara Kw>0$ it
follows as before that $(Kw,\ovl{w})$ is also an rdr field.

Now we discuss the case where $\chara Kw=0$. Then $(K,w)$ is trivially an rdr field, 
and $w_0$ is a coarsening of $v_0$. We write $\ovl{w}=\ovl{w}_0\circ\ovl{w}_p\circ
\ovl{\ovl{w}}$ as for $v$. We obtain that $\ovl{w}_p=v_p\,$, $\ovl{\ovl{w}}=\ovl{v}$, 
and $\ovl{w}_0$ is possibly trivial, with $w\circ\ovl{w}_0=v_0\,$. It follows that
$(Kv_0,v_p)=((Kw)\ovl{w}_0,\ovl{w}_p)$. Using Proposition~\ref{rdrrk1}, we conclude 
that $(K,v)$ is an rdr field if and only if $(Kw,\ovl{w})$ is an rdr field.

It remains to consider the properties ``semitame'' and ``deeply ramified''. We observe
that if $\chara Kv=p>0$, then $vK$ is $p$-divisible if and only if the same is true 
for $wK$ and $\ovl{w}(Kw)$. Likewise, all archimedean components of $vK$ are densely
ordered if and only if the same is true for $wK$ and $\ovl{w}(Kw)$. From what we have
proved before, it thus follows that $(K,v)$ is a semitame (or deeply ramified) field 
if and only if both $(K,w)$ and $(Kw,\ovl{w})$
are semitame (or deeply ramified, respectively).

Further, we recall that in the case of $\chara Kw>0$, $(K,w)$ being an rdr field implies 
that $Kw$ is perfect, and so $\ovl{w}(Kw)$ is $p$-divisible and thus all of its 
archimedean components are densely ordered. This proves that
$(K,v)$ is a semitame (or deeply ramified) field already if $(K,w)$ is.
\qed

%
%
\subsection{Proof of Theorems~\ref{algext} and~\ref{down} for the equal
characteristic case}      
\begin{proposition}                             \label{algextpos}
Take an algebraic extension $(L|K,v)$ of valued fields of positive characteristic. If 
$(K,v)$ is an rdr field, then so is $(L,v)$. If $L|K$ is finite and $(L,v)$
is an rdr field, then so is $(K,v)$. Both statements also hold for ``deeply 
ramified'' and ``semitame'' in place of ``rdr''.
\end{proposition}
\begin{proof}
In view of part 3) of Theorem~\ref{connprop}, our assertions only need to be proved 
for rdr fields. By part 3) of Theorem~\ref{connprop}, a valued field $(K,v)$ of positive 
characteristic is an rdr field if and only if its completion $(\hat K,v)$ is perfect. 

Assume that $(K,v)$ is an rdr field. Then the completion $(\hat L,v)$ 
of $(L,v)$ contains $(\hat K,v)$. Since $\hat K$ is perfect, so is $L.\hat K$.
Since $(\hat L,v)$ is also the completion of $(L.\hat K,v)$, it is perfect too. Hence
$(L,v)$ is an rdr field.

\pars
Now assume that $L|K$ is finite and $(L,v)$ is an rdr field. 
Then $\hat L=L.\hat K$ is perfect. As $L.\hat K|\hat K$ is finite, it follows that 
$\hat K$ is perfect. Thus $(K,v)$ is an rdr field.
\end{proof}

%
%
\subsection{Proof of Theorem~\ref{rdrram} and Corollary~\ref{rdrramcor}}      
\mbox{ }\sn
Our next goal is the proof of Theorem~\ref{rdrram}. First, we prove the upward 
direction. By Proposition~\ref{algextpos}, we only need to prove it in the mixed
characteristic case.
\begin{lemma}                             \label{extrdr}
Assume that $(K,v)$ is a henselian rdr field of mixed characteristic with residue 
characteristic $p>0$, and that $(L|K,v)$ 
is a finite extension. Then the following assertions hold.
\sn
1) If $[L:K]=[Lv:Kv]$, then also $(L,v)$ is an rdr field.
\sn
2) Take a prime $q$ different from $p$. Assume that $L=K(a)$ with $a^q\in K$ and
$va\notin vK$. Then also $(L,v)$ is an rdr field.
\end{lemma}
\begin{proof}
Like $(K,v)$, also $(L,v)$ satisfies (DRvp). Hence by part 3) of 
Lemma~\ref{OK/pOK}, $(L,v)$ will be an rdr field once (\ref{OK/pOKeq}) is surjective.

\pars
In order to prove part 1), we take a finite extension $(L|K,v)$ such that 
$[L:K]=[Lv:Kv]$. Since $Kv$ is perfect by Lemma~\ref{basprop1}, $Lv|Kv$ is separable 
and we write $Lv=Kv(\xi)$ with $\xi\in Lv$. Since also $Lv$ is perfect,
there are $\xi_0,\ldots,\xi_n\in Kv$ with $n=[Lv:Kv]-1$ such that
$\xi=(\xi_n\xi^n+\ldots+\xi_1\xi+\xi_0)^p$. Let $F$ be the extension of $\F_p$ 
generated by the coefficients of the minimal polynomial of $\xi$ over $Kv$ and the
elements $\xi_0,\ldots,\xi_n\,$. As a finitely generated extension of
the perfect field $\F_p\,$, $F$ is separably generated, that is, it admits a
transcendence basis $t_1,\ldots,t_k$ such that $F|\F_p(t_1,\ldots,t_k)$ is 
separable-algebraic. We have that $F\subseteq Kv$, so we may choose
$x_1,\ldots,x_k\in K$ such that $x_iv=t_i\,$. Then $v\Q(x_1,\ldots,x_k)=v\Q=\Z vp$ 
and $\Q(x_1,\ldots,x_k)v=\F_p (t_1,\ldots,t_k)$ (cf.\ \cite[chapter VI, \S10.3,
Theorem~1]{[B]}). Using Hensel's Lemma, we find an extension
$K_0$ of $\Q(x_1,\ldots,x_k)$ within the henselian
field $K$ such that $K_0v=F$ and $vK_0=v\Q(x_1,\ldots,x_k)=\Z vp$.

Using Hensel's Lemma again, we find $a\in L$ such that $av=\xi$, $[K_0(a):K_0]=
[F(\xi):F]$ and $vK_0(a)=vK_0=\Z vp$. By construction, $\xi^{1/p}\in F(\xi)$,
so we can choose $b\in K_0(a)$ such that $bv=\xi^{1/p}$. Then $av=(bv)^p=b^pv$, 
so $v(a-b^p)>0$ and thus $v(a-b^p)\geq vp$.

We observe that since $F$ contains all coefficients of the minimal polynomial of 
$\xi$ over $Kv$,
\[
[Kv(\xi):Kv] \>=\> [F(\xi):F] \>=\> [K_0(a):K_0]\>\geq\> [K(a):K] \>\geq\> 
[Kv(\xi):Kv]\>.
\]
Hence equality holds everywhere; in particular, $K(a)=L$. Also, we obtain that 
$1,a,\ldots,a^n$ is a basis of $K(a)|K$ with the residues $1,av,\ldots,a^nv$ linearly
independent over $Kv$. Hence if we write an arbitrary element of $K(a)$ as
$\sum_{i=0}^n c_i a^i$ with $c_i\in K$, then
\[
v\sum_{i=0}^nc_i a^i \>=\> \min_{0\leq i\leq n} vc_i\>.
\]
Thus, for the sum to have non-negative value, all $c_i$ must have non-negative value.
Since $(K,v)$ is an rdr field, there exists $d_i\in K$ such that $c_i\equiv d_i^p \mod
p\cO_K\,$. Consequently, 
\[
\sum_{i=0}^nc_i a^i \>\equiv\> \sum_{i=0}^n d_i^p (b^p)^i \>\equiv\>
\left(\sum_{i=0}^n d_i b^i\right)^p \mod p\cO_{L}\>,
\]
where the last equivalence holds by part 1) of Lemma~\ref{congp}. This shows that 
(\ref{OK/pOKeq}) is surjective, which proves that $(L,v)$ is an rdr field.

\parm
In order to prove part 2), we take a prime $q$ different from $p$ and a finite 
extension $(L|K,v)$ such that $L=K(a)$ with $a^q\in K$ and $va\notin vK$. We obtain 
that $[K(a):K]=q=(vK(a):vK)$. As $p$ and $q$ are coprime, also
$pva=va^p$ generates $vK(a)$ over $vK$, and $K(a)=K(a^p)$. Therefore, $1,a^p,\ldots,
a^{p(q-1)}$ is a basis of $K(a)|K$ with the values $v1,va^p,\ldots,va^{p(q-1)}$ 
belonging to distinct cosets of $vK$. Hence if we write an arbitrary element $b$ of 
$K(a)$ as $b=\sum_{i=0}^{q-1}c_i a^{pi}$ with $c_i\in K$, then
\[
vb \>=\> v\sum_{i=0}^{q-1}c_i a^{pi} \>=\> \min_{0\leq i<q} vc_i+iva^p\>.
\]
Assume that $vb\geq 0$. Then all $c_ia^{pi}$ must have 
non-negative value. However, for $i>0$ this does not imply that $vc_i\geq 0$; we only 
know that $vc_ia^{pi}>0$ since $iva^p\notin vK$, whence $va^{pi}>-vc_i$.

Suppose that $va$ is not equivalent to an element in $vK$ modulo $(vL)_{vp}\,$. Then 
the same holds for $vc_i +piva$ in place of $va$, for $1\leq i<q$, so that 
$vc_ia^{pi}\notin (vL)_{vp}\,$. In this case, $b$ is equivalent to $c_0$ modulo 
$p\cO_L\,$. Since $(K,v)$ is an rdr field, there is $d_0\in K$ such that $b\equiv
c_0\equiv d_0^p \mod p\cO_L\,$. Hence we may now assume that $va$ is equivalent to an
element $\delta\in vK$ modulo $(vL)_{vp}\,$. We choose $d\in K$ with $vd=\delta$ and
replace $a$ by $a/d$, so from now on we can assume that $va\in (vL)_{vp}\,$.

As $(K,v)$ is an rdr field, $(vK)_{vp}$ is $p$-divisible by part 1) of
Lemma~\ref{basprop2}. It follows that $p(vK)_{vp}$ lies dense in $(vL)_{vp}$ and thus
there is $b_i\in K$ such that $-vc_i\leq pvb_i\leq va^{pi}$, whence $vc_ib_i^p\geq 0$ 
and $vb_i^{-p}a^{pi}\geq 0$. Again since $(K,v)$ is an rdr field, there are $d_i\in K$
such that
$c_ib_i^p\equiv d_i^p \mod p\cO_K\,$. Hence we obtain that
\[
\sum_{i=0}^{q-1}c_i a^{pi} \>=\> \sum_{i=0}^{q-1}(c_ib_i^p) (b_i^{-p}a^{pi}) 
\>\equiv\> \sum_{i=0}^{q-1}d_i^p
b_i^{-p}a^{pi} \>\equiv\> \left(\sum_{i=0}^{q-1}d_i b_ia^i\right)^p \mod p\cO_L\>,
\]
where the last equivalence holds by part 1) of Lemma~\ref{congp}. Again, this shows that 
(\ref{OK/pOKeq}) is surjective, which proves that $(L,v)$ is an rdr field.
\end{proof}

\begin{proposition}                              \label{rdrupmc}
Take a valued field $(K,v)$ of mixed characteristic, fix any extension of $v$ to 
$\tilde{K}$, and let $(K^r,v)$ be the corresponding absolute ramification field of 
$(K,v)$. If $(K,v)$ is an rdr field, then so is $(K^r,v)$.
\end{proposition}
\begin{proof}
In this proof we will freely make use of facts from ramification theory; for details, 
see \cite{[En],[EP],[K7]}.

We let $L$ be a maximal extension of $K$ inside of $K^r$ that is again an rdr field; 
since the union over an ascending chain of rdr fields is again an rdr field, $L$ 
exists by Zorn's Lemma.

First we will show that $(L,v)$ is henselian. The decomposition $v=v_0\circ v_p\circ 
\ovl{v}$ for $v$ on $K$ carries over to $L$ with extensions of the 
respective valuations $v_0\,$, $v_p$ and $\ovl{v}$. We note that 
$v$ is henselian on $L$ if and only if $v_0$, $v_p$ and $\ovl{v}$ are.

Suppose that $v_0$ is not henselian on $L$. As $(K^r,v)$ is henselian, 
so is $(K^r,v_0)$ which therefore contains a
henselization $L^{h(v_0)}$ of $L$ with respect to $v_0$. As henselizations are 
immediate extensions, we know that
$L^{h(v_0)}v_0=Lv_0\,$; by Proposition~\ref{rdrrk1}, $(Lv_0,v_p)$ is an rdr field. 
Using the same proposition again,
we find that also $(L^{h(v_0)},v_0)$
is an rdr field. By the maximality of $L$ we conclude that $L^{h(v_0)}=L$, so $v_0$ is henselian on $L$.

\pars
Next, suppose that $v_p$ is not henselian on $Lv_0\,$. As $(K^r,v)$ is henselian, 
so is $(K^r v_0, v_p)$ which therefore contains a henselization $Lv_0^{h(v_p)}$ of 
$Lv_0$ with respect to $v_p$. We know already that 
$(Lv_0,v_p)$ is an rdr field. As its rank is 1, its henselization lies in its
completion. Hence by part 4) of Lemma~\ref{OK/pOK}, $(Lv_0^{h(v_p)},v_p)$ satisfies
(DRvr). Since (DRvp) holds in $(Lv_0,v_p)$, it also holds in $(Lv_0^{h(v_p)},v_p)$, 
so the latter is an rdr field. The extension $Lv_0^{h(v_p)}|Lv_0$ is 
separable-algebraic, so we can use Hensel's Lemma to find an extension $L'$
of $L$ within $K^r$ such that $L'v_0=Lv_0^{h(v_p)}$. Using Proposition~\ref{rdrrk1}
again, we find that $(L',v)$ is an rdr field. Hence $L'=L$ by the maximality of $L$, 
that is, $Lv_0=Lv_0^{h(v_p)}$, showing that $(Lv_0,v_p)$ is henselian.
\pars
Finally, suppose that $\ovl{v}$ is not henselian on $Lv_0v_p\,$. As $(K^r,v)$ is
henselian, so is $(K^r v_0v_p,\ovl{v})$ which therefore contains a henselization
$Lv_0v_p^{h(\ovl{v})}$ of $Lv_0v_p$ with respect to $\ovl{v}$. Suppose that
$Lv_0v_p^{h(\ovl{v})}|Lv_0v_p$ is nontrivial, so it contains a finite separable
subextension. Using Hensel's Lemma, we lift it to a subextension $F|L$ of $K^r|L$ such
that $[F:L]=[Fv_0v_p:Lv_0v_p]$. By what we have shown already, $(L,v_0v_p)$ is 
henselian, and by definition it is of mixed characteristic. Therefore, we can employ 
part 1) of Lemma~\ref{extrdr} to deduce that $(F,v_0v_p)$ is an rdr field. By
Proposition~\ref{rdrrk1}, also $(F,v)$ is an rdr field. This contradiction to the
maximality of $L$ shows that $Lv_0v_p^{h(\ovl{v})}=Lv_0v_p\,$, that is,
$(Lv_0v_p,\ovl{v})$ is henselian. Altogether, we have now shown that $(L,v)$ is
henselian.

\pars
The residue field of $K^r$ is the separable-algebraic closure of $Kv$. Suppose that 
$Lv$ is not separable-algebraically closed, so it admits a finite separable-algebraic
extension. Using Hensel's Lemma, we lift it to a subextension $F|L$ of $K^r|L$ such 
that $[F:L]=[Fv:Lv]$. Again by part 1) of Lemma~\ref{extrdr}, $(F,v)$ is an rdr field,
contradicting the maximality of $L$. Hence $Lv$ is separable-algebraically closed.

\pars
The value group of $K^r$ is the closure of $vK$ under division by all primes other 
than $p$. Suppose that $vL\ne vK^r$. Then there is some prime $q\ne p$ and $\alpha
\in vK^r\setminus vL$ with $q\alpha\in vL$. Take $a\in\tilde K$ such that $a^q\in L$
with $va^q=q\alpha$. It follows that $(L(a)|L,v)$ is a tame extension, hence $a$ lies 
in the maximal tame extension $L^r$ of $L$. Since $K\subseteq L\subset K^r$, we 
know that $K^r=L^r$, so $a\in K^r$. By part 2) of Lemma~\ref{extrdr}, also $(L(a),v)$ 
is an rdr field, which again contradicts the maximality of $(L,v)$. We conclude
that $vL=vK^r$.

\pars
By what we have shown, $Lv=K^rv$ and $vL=vK^r$. As $K^r=L^r$, we see that $(K^r|L,v)$ 
is a tame extension. Together with the equality of the value groups and residue 
fields, this implies that $L=K^r$. Thus $(K^r,v)$ is an rdr field.
\end{proof}

\begin{proposition}                                \label{rdrv(a-Kp)}
Assume that $(K,v)$ is an rdr field of mixed characteristic, and take $a\in \cO_K\,$. 
\sn
1) Assume that $va=0$. Then for every $c\in\cO_K$ with $0<v(a-c^p)\in (vK)_{vp}$ there 
is $c_1\in\cO_K$ such that
\[
v(a-c_1^p)\>=\> vp\,+\,\frac{1}{p}v(a-c^p)\>.
\]
2) Assume that $va\in (vK)_{vp}$ and that $\dist(a,K^p)< va+\frac{p}{p-1} vp$. Then
\[
va+vp\><\>\dist(a,K^p) \>=\> va+\frac{p}{p-1} vp \,+\, H^-\>,
\]
where $H$ is a convex subgroup of $vK$ not containing $vp$.
\end{proposition}
\begin{proof}
1) Set $\alpha:=v(a-c^p)>0$. Since $(K,v)$ is an rdr field, part 2) of
Lemma~\ref{basprop2} shows that there is $\tilde{c}\in K$ such that:
\begin{equation}                           \label{eq1}
v(a-c^p - \tilde{c}^p)\>\geq\> vp+\alpha\>.
\end{equation}
It follows that $v\tilde{c}^p=\alpha>0$. Since $vc=va=0$,
\begin{equation}                           \label{eq2}
v((c+\tilde{c})^p - c^p-\tilde{c}^p) \>=\> v\sum_{i=1}^{p-1} \binom{p}{i} c^{p-i}
\tilde{c}^i \>=\> vp+v\tilde{c} \>=\> vp+\frac{\alpha}{p}\>.
\end{equation}
From (\ref{eq1}) and (\ref{eq2}), we obtain for $c_1:=c+\tilde{c}$:
\[
v(a-c_1^p) \>=\> \min\left\{vp+\alpha,vp+\frac{\alpha}{p}\right\} \>=\> 
vp+\frac{\alpha}{p}\>.
\]
\sn
2) First we prove the assertion in the case of $va=0$. Since $(K,v)$ is an rdr field,
there is some $c\in K$ such that $v(a-c^p)\geq vp$, so $\dist(a,K^p)\geq vp$.

We will use the following observation. If $(vK)_{vp}\ni v(a-c^p)\geq\frac{p}{p-1} vp 
- \varepsilon>0$ for some $c\in K$ and  positive $\varepsilon\in vK$, then by part 1)
there is $d\in\cO_K$ such that
\[
v(a-d^p)\>=\> vp+\frac{v(a-c^p)}{p}\>\geq\> vp+\frac{vp}{p-1} - \frac{1}{p}\,
\varepsilon \>=\> \frac{p}{p-1} vp - \frac{1}{p}\,\varepsilon\>.
\]

By assumption, $\dist(a,K^p)< \frac{p}{p-1} vp$. Hence the set of all convex 
subgroups $H'$ of $vK$ such that $v(a-K^p)<\frac{p}{p-1}
vp+H'$ is nonempty as it contains $\{0\}$. The set is
closed under arbitrary unions, so it contains a maximal subgroup $H$.
Since $0\in  v(a-K^p)$, we see that $H$ cannot contain $vp$.

Take any positive $\delta\notin H$. Then by the definition of $H$, there is some 
$n\in\N$ such that $v(a-K^p)$ contains a value $\geq \frac{p}{p-1} vp -n\delta$. 
We set $\varepsilon:=\min\left\{\frac{p}{p-1} vp -vp, n\delta
\right\}$ and observe that there is $c\in K$ such that
\[
v(a-c^p)\>\geq\> \frac{p}{p-1} vp - \varepsilon \>\geq\> vp\> >\>0\>.
\]
Note that $v(a-c^p)\in (vK)_{vp}$ since $\dist(a,K^p)< \frac{p}{p-1} vp$. Using our 
above observation, by induction starting from $c_0=c$ we find $c_i\in K$ such that
\[
v(a-c_i^p)\>\geq\> \frac{p}{p-1} vp - \frac{1}{p^i}\varepsilon\>.
\]
We choose some $j\in\N$ such that $\frac{n}{p^j}<1$. Then
\[
\frac{1}{p^j}\varepsilon\leq \frac{n}{p^j}\delta<\delta
\]
and consequently,
\[
v(a-c_j^p)\> > \> \frac{p}{p-1} vp -\delta\>.
\]
This together with the definition of $H$ shows that
\begin{equation}               \label{eq3}
vp\><\>\dist(a,K^p) \>=\> \frac{p}{p-1} vp \,+\, H^- \>.
\end{equation}
If $0\ne va\in (vK)_{vp}$, then since $(K,v)$ is an rdr field, part 1) of
Lemma~\ref{basprop2} shows that there is $b\in K$ such that $vb^p=va$. By what we have
already shown, (\ref{eq3}) holds for $b^{-p}a$ in place of $a$. We have that
\[
v(a-(bc)^p)\>=\>vb^p+v(b^{-p}a-c^p)\>=\>va+v(b^{-p}a-c^p)\>,
\]
whence
\[
\dist(a,K^p)\>=\> va+\dist(b^{-p}a,K^p)\>,
\]
which together with (\ref{eq3}) for $b^{-p}a$ in place of $a$ proves assertion 2) 
of our proposition.
\end{proof}

We pause to note the following consequence of Proposition~\ref{rdrv(a-Kp)} which was
mentioned in the Introduction, but will not be needed any further.
\begin{proposition}                                   \label{dpd'2}
Take a valued field $(K,v)$ of mixed characteristic such that $(vK)_{vp}$ 
is $p$-divisible. Further, take $d'\in \cM_K$ such that
$vp\leq vd'<\frac{p}{p-1}vp+H_0^-$ for the largest convex
subgroup $H_0$ of $vK$ not containing $vp$. Then
the following assertions are equivalent:
\sn
a) the function (\ref{OK/pOKeq}) is surjective, so $(K,v)$ is an rdr field,
\sn
b) the function (\ref{OK/dOKeq}) is surjective modulo $d'\cO_K\,$.
\end{proposition}
\begin{proof}
Since $d'\cO_K\subseteq p\cO_K\,$, and in
view of the equivalence of a) and b) in part 3) of Lemma~\ref{OK/pOK}, the 
implication b)$\Rightarrow$a) is trivial.
\sn
a)$\Rightarrow$b): Assume that assertion a) holds, and take any $a\in\cO_K\,$. 
We may assume that $va\in (vK)_{vp}$ since otherwise, $va>vd'$ and there is nothing to 
show. By our choice of $H_0$ and part 2) of Proposition~\ref{rdrv(a-Kp)}, we now obtain:
\[
\dist(a,K^p) \>\geq\> va+\frac{p}{p-1} vp \,+\, H_0^-
\>\geq\> \frac{p}{p-1} vp \,+\, H_0^-\>.
\]
Therefore, there is $c\in K$ such that $v(a-c^p)\geq vd'$. 
This proves assertion b). 
\end{proof}

\pars
The next two propositions will describe the relation between rdr and independent 
defect fields.
\begin{proposition}                                \label{rdr->id}
Every rdr field containg all $p$-th roots of unity is an independent defect field.
\end{proposition}
\begin{proof}
Assume first that $\chara K>0$. Then by part 3) of Theorem~\ref{connprop}, the perfect
hull of $(K,v)$ lies in its completion. Now part 2) of Proposition~\ref{idf} 
(which has already been proved at the end of Section~\ref{sectASde})
shows that $(K,v)$ is an independent defect field.
\pars
Now assume that $\chara K=0$, and take a Galois defect extension $(L|K,v)$ of 
prime degree. As shown in the beginning of Section~\ref{sectGaldefdegp}, we can assume
that $L=K(\eta)$ with a Kummer generator $\eta$ which is a $1$-unit. 

Suppose that there is some $c\in K$ such that $v(\eta-c)\geq \frac{vp}{p-1}$. Since 
the defect extension $(K(\eta)|K,v)$ is immediate, $v(\eta-c)$ has no maximal element,
and so there will also be some element $c\in K$ such that $v(\eta-c)> \frac{vp}{p-1}$.
Then by Lemma~\ref{apKras}, $\eta$ lies in some henselization $K^h$. But this is
impossible since by Lemma~\ref{dist_hens}, the unibranched extension $(K(\eta)|K,v)$ is
linearly disjoint from $K^h|K$. We conclude that $v(\eta-K)<\frac{vp}{p-1}$. By
Lemma~\ref{eta-c}, this is equivalent to $v(\eta^p-K^p)<\frac{p}{p-1}vp$. 
Therefore, we can apply part 2) of Proposition~\ref{rdrv(a-Kp)} to $a=\eta^p$. We 
find that
\begin{equation}                    \label{distetap}
\dist(\eta^p,K^p) \>=\> \frac{p}{p-1} vp \,+\, H^-\>,
\end{equation}
where $H$ is a convex subgroup of $vK$ not containing $vp$. By part 1) of 
Lemma~\ref{basprop2}, $(vK)_{vp}$ is $p$-divisible. Since $H\subset (vK)_{vp}$, we can 
again apply Lemma~\ref{eta-c} to obtain that (\ref{distetap}) is equivalent to
(\ref{dist=+H-}). 
By part 3) of Proposition~\ref{distAS} it follows that $(K(\eta)|K,v)$ 
has independent defect. This proves that $(K,v)$ is an independent defect field.
\end{proof}

\begin{proposition}                                   \label{chardr2}
Assume that $(vK)_{vp}$ is $p$-divisible and $\erf(K,v)$ is perfect. If $(K,v)$ is an
independent defect field, then it is an rdr field.
\end{proposition}
\begin{proof}
From our assumption that $(vK)_{vp}$ is $p$-divisible it follows that (DRvp) holds. 
It remains to show that $(K,v)$ satisfies (DRvr).

\pars
Assume first that $\chara K>0$. Then by assumption, $vK$ is $p$-divisible and $Kv$ is
perfect, hence the perfect hull of $K$ is an immediate extension of $(K,v)$. Thus by
part 2) of Proposition~\ref{idf}, our
assumption that $(K,v)$ is an independent defect field implies that the perfect hull
of $K$ lies in its completion. This means that 
$(K,v)$ lies dense in its perfect hull. Now part 3) of Theorem~\ref{connprop} shows 
that $(K,v)$ is an rdr field.

\pars
Now assume that $\chara K=0$, and set $w:=v_0\circ v_p$. By Proposition~\ref{wcow} 
it suffices to prove that $(K,w)$ is an rdr field. Assume that $b\in K$ is not a $p$-th
power, and take $\eta\in\tilde{K}$ with $\eta^p=b$. Then from Lemma~\ref{maxnots} with 
$w$ in place of $v$ we infer that either $w(\eta-K)$ has 
a maximal element $\geq\frac{wp}{p}$, or it has no maximal element at all. In the 
first case, part 3) of Lemma~\ref{congp} shows the existence of $c\in K$ such
that $b\equiv c^p \mod p\cO_{(K,w)}\,$. 

Assume that $w(\eta-K)$ has no maximal element. If it is not bounded from above in 
$(wK)_{wp}\,$, then there is some $c\in K$ such that $w(\eta-c)\geq\frac{wp}{p}$, 
which by part 3) of Lemma~\ref{congp} gives us that $b\equiv c^p \mod p\cO_{(K,w)}\,$. 

Now assume that $w(\eta-K)$ is bounded from above in $(wK)_{wp}\,$. Then in particular, 
$w\eta\in (wK)_{wp}\,$. It follows that $(\eta v_0)^p=bv_0\in Kv_0$ and that 
$v_p(\eta v_0 - Kv_0)$ has no maximal element but is bounded from above in 
$v_p(Kv_0)=(wK)_{wp}\,$. Hence by Lemma~\ref{ude}, $(Kv_0(\eta v_0)|Kv_0,v_p)$ is a 
defect extension of degree $p$. From this it follows that also $(K(\eta)|K,v)$ is a 
defect extension of degree $p$. We set $K':=K(\zeta_p)$ where $\zeta_p$ is a 
primitive $p$-th root of unity. Then by (\ref{distKN}) of Proposition~\ref{K(a)K^r(a)}, 
also $(K'(\eta)|K',v)$ is a defect extension of degree $p$, with $\dist(\eta,K')
=\dist(\eta,K)$. By assumption, this defect extension is independent, so 
\[
\dist(\eta,K)\>=\>\dist(\eta,K')\>=\>\frac{vp}{p-1}\,+\,H^-
\]
for some  proper convex subgroup $H$ of $vK$ with $vp\notin H$. Hence there is 
some $c\in K$ such that $v(\eta-c)\geq\frac{vp}{p}$; thus as before, $b\equiv c^p \mod 
p\cO_K\,$. This implies that $b\equiv c^p \mod p\cO_{(K,w)}\,$. 

Altogether, we have shown that (\ref{OK/pOKeq}) is surjective. Hence by part 2) of
Lemma~\ref{OK/pOK}, (DRvr) holds.
\end{proof}

\begin{lemma}                                      \label{Kr=>K}
Fix any extension of $v$ from $K$ to $\tilde{K}$, and let $(K^r,v)$ be the corresponding
absolute ramification field of $(K,v)$. If $(K^r,v)$ is an rdr field, then so is 
$(K,v)$, and if $(K^r,v)$ is a semitame field, then so is $(K,v)$.
\end{lemma}
\begin{proof}
Assume that $(K^r,v)$ is an rdr field and hence an independent defect field by
Proposition~\ref{rdr->id}. By Lemmas~\ref{basprop1} and~\ref{basprop2}, $(vK^r)_{vp}$ 
is $p$-divisible and $\erf(K^r,v)$ is perfect. Since $vK^r/vK$ has no 
$p$-torsion, it follows that $(vK)_{vp}$ is $p$-divisible. From Lemma~\ref{Krw|Kwsep}
we infer that the extension $\erf(K^r,v)\,|\,\erf(K,v)$ is separable, so $\erf(K,v)$ is 
perfect. We set $K':=K(\zeta_p)\subseteq K^r$ as before.
From Proposition~\ref{idfKK^r} we conclude that $(K',v)$ is an independent defect
field. Hence by definition, the same holds for $(K,v)$. Proposition~\ref{chardr2} now 
shows that $(K,v)$ is an rdr field.

Now assume that $(K^r,v)$ is a semitame field. Then by  
Theorem~\ref{connprop}, $(K^r,v)$ is an rdr field, hence so is $(K,v)$. Since $vK^r$ is
$p$-divisible and $vK^r/vK$ has no $p$-torsion, also $vK$ is $p$-divisible. Hence by
definition, $(K^r,v)$ is a semitame field.
\end{proof}

\mn
{\it Proof of Theorem~\ref{rdrram}:}\n
It has been proven already in Lemma~\ref{Kr=>K} that if $(K^r,v)$ is an rdr field, 
then so is $(K,v)$, and if $(K^r,v)$ is a semitame field, then so is $(K,v)$. Let us 
now assume that $(K,v)$ is an rdr field. If $\chara K>0$, then $(K,v)$ is an rdr field 
by Proposition~\ref{algextpos}. The case of rdr fields of mixed characteristic has been
settled in Proposition~\ref{rdrupmc}. Being an rdr field,
$(K^r,v)$ is also deeply ramified, as its value group is divisible by every prime
$q\ne\chara Kv$ and thus satisfies (DRvg).

\pars
Assume now that $(K,v)$ is a semitame field. Then by part 1) of Theorem~\ref{connprop}, 
$(K,v)$ is an rdr field. As shown above, it follows that the same is true for 
$(K^r,v)$. Since $vK$ is $p$-divisible, $vK^r$ is $p$-divisible too. Hence 
$(K^r,v)$ is a semitame field.
\qed

\mn
{\it Proof of Corollary~\ref{rdrramcor}:}\n
Part 1) is an immediate consequence of both the upward and the downward direction of 
Theorem~\ref{rdrram}. As $(K^h,v)$ is a subextension of $(K^r,v)$, the assertions of
part 2) for rdr and semitame fields follow immediately from part 1). Also the 
assertion for the case of deeply ramified fields follows since the extension 
$(K^h|K,v)$ is immediate, so $(K^h,v)$ satisfies (DRvg) if and only if $(K,v)$ does.
\qed

%
%
\subsection{Proof of Theorem~\ref{algext}}         
\mbox{ }\sn
We will need some preparations.
\begin{proposition}                      \label{rdrupdefGe}
Assume that $(K,v)$ is an rdr field of mixed characteristic containing 
all $p$-th roots of unity, and that $(L|K,v)$ 
is a Galois defect extension of prime degree. Then also $(L,v)$ is an rdr field. 
\end{proposition}
\begin{proof}
Let $p$ be the residue
characteristic of $(K,v)$. By part 1) of Lemma~\ref{basprop2}, $vK$ contains 
$\frac{1}{p^\infty}\Z vp$. We choose $d\in K$ such that 
\[
vd\>=\> \frac{vp}{p}\>.
\]
By Proposition~\ref{dpd'1} with $L$ in place of $K$, in order to show that $(L,v)$ is 
an rdr field, it suffices to show that the function $\cO_L\ni x\mapsto x^p\in \cO_L$ 
is surjective modulo $d\cO_L\,$.

From Section~\ref{sectGaldefdegp} we know that the extension admits a Kummer generator
which is a $1$-unit $1+a$ with $a\in\cM_L\,$. Proposition~\ref{rdr->id} shows
that $(K,v)$ is an independent defect field. By Proposition~\ref{distAS}, 
$\dist(1+a,K)=\frac{vp}{p-1}+H^-$ for some convex subgroup $H$ of $vK$ 
that does not contain $vp$. Hence for every positive $\alpha<\frac{vp}{p-1}$ in
$\frac{1}{p^\infty}\Z vp$ there is some $b\in K$ such that $v(1+a-b)\geq\alpha$. 
Then $b$ must itself be a $1$-unit, say $b=1+c$. Now $v(1+a -(1+c))=v(a-c)$ and
the $1$-unit 
\[
1+a_c \>:=\>\frac{1+a}{1+c}\>=\> 1+(a-c)-\frac{c(a-c)}{1+c}
\]
satisfies $va_c =v(a-c)\geq\alpha$. Since $b\in K$, $1+a_c$ is also a 
Kummer generator of the extension.

We note that $\frac{1}{p^\infty}\Z vp$ is dense in $\Q vp$. Thus we can
choose $\alpha$ so close to $\frac{vp}{p-1}$ that 
\begin{equation}                          \label{vd}
vd \>=\> \frac{vp}{p}\>\leq\> \alpha-2p\left(\frac{vp}{p-1}-\alpha\right) \><\>\alpha
\><\> \frac{vp}{p-1} \>\leq\>vp\>.
\end{equation}
By what we have shown above, we may from now on assume that $L|K$ admits a Kummer
generator which is a $1$-unit $\eta=1+a$ with $va\geq\alpha$.

\pars
Take an element in $\cO_L$ and write it as $f(\eta)$ where $f(X)=\sum_{i=0}^{p-1} c_i 
X^i$ with $c_i\in K$. The problem is that even though $f(\eta)$ lies in $\cO_L$, the 
coefficients $c_i$ do not necessarily lie in $\cO_K$. (This is in contrast to the case 
of defectless extensions, such as extensions within the absolute ramification field, 
where for a suitably chosen $\eta$, the value of $f(\eta)$ is equal to the minimum of 
the values of the summands $c_i\eta^i$.) Since $v(\eta-K)$ has no maximal element,
Lemma~\ref{Kap} shows that there must be some 
$\gamma\in v(\eta-K)$ such that for all $b\in K$ with $v(\eta-b)\geq\gamma$, the
monomials $\partial_i f(b)(\eta-b)^i$ in 
\[
f(\eta)\>=\>\sum_{i=0}^{p-1} \partial_i f(b)(\eta-b)^i 
\] 
have distinct and thus non-negative values, and that for each $i$, the values 
$v\partial_i f(b)$ are constant, say equal to $\beta_i$. Consequently,
\[
\beta_i + i\gamma\> \geq\> 0 \;\mbox{ for } 0\leq i\leq p-1\>.
\]
As all of this will remain true if we replace $\gamma$ by any larger value in 
$v(\eta-K)$, we can assume that $\gamma>va\geq\alpha>0$. We fix one $b$ with 
$v(\eta-b)\geq\gamma$.
Then also $b$ must be a $1$-unit, and we write $b=1+c$ with $c\in \cM_K\,$. Thus, 
$v(a-c)=v(\eta-b)\geq\gamma$, and it follows that $vc=va\geq\alpha$. 
We set
\[
\eta_c\>:=\>\frac{1+a}{1+c}
\]
and observe that
\begin{equation}
a-c\>=\> \eta_c - 1 + \frac{c(a-c)}{1+c} \>\equiv\> \eta_c - 1 \mod c(a-c)\cO_L\>.
\end{equation}

We choose some $z\in K$ with value $vz=v(a-c)$. Then
\begin{equation}                              \label{eqv1}
f(\eta)\>=\> \sum_{i=0}^{p-1} \partial_i f(b)(a-c)^i \>=\> \sum_{i=0}^{p-1} 
\partial_i f(b)z^i \left(\frac{\eta_c - 1}{z}+\frac{c(a-c)}{z(1+c)}\right)^i\>.
\end{equation}
Now $\partial_i f(b)z^i\in\cO_K$ for all $i$. Further, $v\frac{\eta_c - 1}{z}=0$ and 
$\frac{c(a-c)}{z(1+c)}\in c\cO_L$. Consequently, 
\begin{equation}                              \label{eqv2}
f(\eta) \>\equiv\> \sum_{i=0}^{p-1} \partial_i f(b)z^i 
\left(\frac{\eta_c - 1}{z}\right)^i \mod c\cO_L\>.
\end{equation}
Since $vc=va\geq\alpha>vd$, this congruence also holds modulo $d\cO_L\,$.
Hence in order to show that $f(\eta)$ is a $p$-th power in $\cO_L$ modulo $d\cO_L\,$,
it suffices to show that this is true for the polynomial on the right hand side
of (\ref{eqv2}). With the element $C$ as in Lemma~\ref{C}, this polynomial is equal to 
\begin{equation}                              \label{eqv3}
\sum_{i=0}^{p-1} \partial_i f(b)C^i \left(\frac{\eta_c - 1}{C}\right)^i
\>=\> \sum_{i=0}^{p-1} \partial_i f(b)C^i\vartheta^i\>,
\end{equation}
where 
\[
\vartheta \>:=\> \vartheta_{\eta_c}\>=\> \frac{\eta_c - 1}{C}\>.
\]
Since 
\[
vC\>=\>\frac{vp}{p-1}\> >\>v(\eta-b)\>=\>v(a-c)\>=\>vz\>,
\]
the coefficients $\partial_i f(b)C^i$ still lie in $\cO_K\,$.
We note that $v\vartheta=v(\eta_c - 1)-vC=v(a-c)-vC<0$. Further,
\begin{equation}                              \label{vtheta}
0\> >\> v\vartheta \>\geq\>\gamma-vC\> >\>\alpha-vC\>.
\end{equation}

\pars
From (\ref{trpol}) and (\ref{p-part}) we know that $\vartheta$ satisfies the equation 
\[
\vartheta\>=\> \vartheta^p-\frac{\eta_c^p-1}{C^p} +g(\vartheta)\>,
\]
where
\[
g(\vartheta)=\sum_{i=2}^{p-1} \binom{p}{i} C^{i-p}\vartheta^i\>.
\]
We compute for $2\leq i\leq p-1$:
\begin{eqnarray*}
v\binom{p}{i} C^{i-p}\vartheta^i &=& vp+(i-p)vC+iv\vartheta \>=\>(i-1)vC+iv\vartheta\\
&\geq& vC +pv\vartheta \> > \> \alpha+pv\vartheta \>=\> \alpha
+2pv\vartheta-pv\vartheta\\
&>& \alpha -2p(vC-\alpha)-pv\vartheta\>\geq\> vd-pv\vartheta \>,
\end{eqnarray*}
where the last inequality holds by (\ref{vd}). Hence $g(\vartheta)\in
d\vartheta^{-p}\cO_L$ and
\begin{equation}                              \label{eqv5}
\vartheta\>\equiv\> \vartheta^p-\frac{\eta_c^p-1}{C^p} \mod d\vartheta^{-p}\cO_L\>.
\end{equation}
As $v\vartheta<0$, we have that $v\vartheta>v\vartheta^p$, so that 
\[
v\frac{\eta_c^p-1}{C^p}\>=\>v\vartheta^p\>.
\]
Since $(K,v)$ is an rdr field, using part 2) of Lemma~\ref{basprop2} we can find elements 
$t,t_i\in \cO_K$ such that 
\begin{equation}                                   \label{tt_i1}
t^p\>\equiv\>\frac{\eta_c^p-1}{C^p} \mod pt^p\cO_K\>=\> p\vartheta^p\cO_L
\end{equation}
and for $0\leq i\leq p-1$,
\begin{equation}                                   \label{tt_i2}
t_i^p \>\equiv\> \partial_i f(b)C^i \mod p\cO_K\>\subseteq\> d\cO_L\>.
\end{equation}
We have that 
\[
\vartheta^p-t^p\>\equiv\>(\vartheta-t)^p \mod p\vartheta^p\cO_L
\]
and consequently,  
\begin{equation}                              \label{eqv6}
\vartheta^p-\frac{\eta_c^p-1}{C^p}\>\equiv\>(\vartheta-t)^p \mod p\vartheta^p\cO_L\>.
\end{equation}
From (\ref{vd}) and (\ref{vtheta}) we derive that 
\[
vd\vartheta^{-2p}\>=\> vd-2pv\vartheta\><\>vd+2p(vC-\alpha)\>\leq\>\alpha \><\>vp\>,
\]
so that $p\vartheta^p\cO_L\subseteq d\vartheta^{-p}\cO_L\,$. Hence by 
(\ref{eqv5}) and (\ref{eqv6}),
\[
\vartheta\>\equiv\> (\vartheta-t)^p \mod d\vartheta^{-p}\cO_L\>.
\]
We write $(\vartheta-t)^p=\vartheta+d\vartheta^{-p}s$ with $s\in\cO_L\,$. Then
for $0\leq i\leq p-1$,
\[
(\vartheta-t)^{ip}\>=\> \vartheta^i\,+\,\sum_{j=1}^i \binom{i}{j} \vartheta^{i-j}
(d\vartheta^{-p}s)^j\>.
\]
Since $v\vartheta<0<vd\vartheta^{-p}s$, the summand of least value in the sum on the 
right hand side is the one for $j=1$. This shows that
\[
\vartheta^i \>\equiv\> (\vartheta-t)^{ip} \mod d\vartheta^{-p+i-1}\cO_L\>.
\]
Here, each $d\vartheta^{-p+i-1}\cO_L$ can be replaced by the larger ideal 
$d\cO_L$. Combining this with (\ref{tt_i2}), we obtain:
\begin{equation}                              \label{eqv4}
\sum_{i=0}^{p-1} \partial_i f(b)C^i\vartheta^i \>\equiv\> \sum_{i=0}^{p-1}t_i^p
(\vartheta-t)^{ip} \mod d\cO_L\>.
\end{equation}
We observe that the corresponding summands in the sums on the right hand sides of 
(\ref{eqv1}), (\ref{eqv2}), (\ref{eqv3}) and (\ref{eqv4})
all have the same non-negative value. Consequently, 
\[
\sum_{i=0}^{p-1}t_i^p (\vartheta-t)^p \>\equiv\> 
\left(\sum_{i=0}^{p-1}t_i (\vartheta-t)\right)^p\mod p\cO_L\>.
\]
Together with (\ref{eqv4}), this leads to 
\[
f(\eta)\>\equiv\> \sum_{i=0}^{p-1} \partial_i f(b)C^i \left(\frac{\eta_c - 1}{C}\right)^i
\>\equiv\> \left(\sum_{i=0}^{p-1}t_i (\vartheta-t)\right)^p\mod d\cO_L\>,
\]
which completes our proof.
\end{proof}

\begin{proposition}                                   \label{rdrupdlp}
Assume that $(K,v)$ is an rdr field of mixed characteristic with algebraically closed
residue field. Take a defectless unibranched Galois extension $(L|K,v)$ of degree 
$p=\chara Kv$. Then also $(L,v)$ is an rdr field. 
\end{proposition}
\begin{proof}
Since $(L|K,v)$ is unibranched and defectless, equation (\ref{feuniq}) shows that
$p=[L:K]=(vL:vK)[Lv:Kv]$.
However, as $Kv$ is algebraically closed, $[Lv:Kv]=1$. Hence $(vL:vK)=p$. 
By part 1) of Lemma~\ref{basprop2}, $v_p\circ\ovl{v}(Kv_0)=(vK)_{vp}$ is 
$p$-divisible. It follows that $(v_0L:v_0K)=p$ and therefore, $Lv_0=Kv_0$. 
Applying Proposition~\ref{rdrrk1} to $(K,v)$, we find that $(Kv_0,v_p)=(Lv_0,v_p)$ 
is an rdr field,
and applying the proposition again, we conclude that $(L,v)$ is an rdr field.
\end{proof}

\bn
{\it Proof of Theorem~\ref{algext}.} 
For the case of deeply ramified fields of positive characteristic we have given 
the proof already in Proposition~\ref{algextpos}, so let us assume that $(K,v)$ is an
rdr field of mixed characteristic and $(L|K,v)$ an algebraic extension.
By Theorem~\ref{rdrram}, $(K^r,v)$ is a deeply ramified field. Hence $K^rv$ is perfect
by Lemma~\ref{basprop1}, but as it is also separable-algebraically closed, it must be 
algebraically closed.

We let $L'$ be a maximal 
extension of $K^r$ inside of $L.K^r$ that is again an rdr field; since the union over 
an ascending chain of rdr fields is again an rdr field, $L'$ exists by Zorn's Lemma.
Since $K^r$ contains all $p$-th roots of unity, so does $L'$, and since $K^rv$ is 
algebraically closed, so is $L'v$.

Suppose that $L'\ne L.K^r$. Since $\tilde{K}|K^r$ is a $p$-extension, the same holds 
for $\tilde{K}|L'$. Consequently, $L.K^r|L'$ contains a Galois subextension 
$(L''|L',v)$ of degree $p$. If this is a defect extension, then it follows from 
Proposition~\ref{rdrupdefGe} that $(L'',v)$ is an rdr field. If the extension is 
defectless, then it follows from Proposition~\ref{rdrupdlp} that $(L'',v)$ is an rdr
field. In both cases we have obtained a contradiction to the maximality of $L'$. This
proves that $(L.K^r,v)$ is an rdr field. Since $L.K^r=L^r$ by \cite[(20.15) b)]{[En]}, 
we now obtain from Theorem~\ref{rdrram} that $(L,v)$ is an rdr field.

\pars
It remains to deal with deeply ramified fields and with semitame fields. For them the 
proof follows immediately from what we have already shown, since deeply ramified 
fields are just the rdr fields that satisfy (DRvg), and semitame fields are
just the rdr fields with $p$-divisible value groups. All of these properties are
preserved under algebraic extensions.
\qed

%
%
\subsection{Proof of Theorem~\ref{down}}         
\mbox{ }\sn
The equal characteristic case has already been settled in Proposition~\ref{algextpos}.
Thus we assume now that $(L|K,v)$ is a finite extension of valued fields of mixed
characteristic and that $(L,v)$ is an rdr field. We wish to show that $(K,v)$ is an rdr
field. In order to derive a contradiction, we suppose that this is not the case.

We take an extension of $v$ to $\tilde{K}=\tilde{L}$. This determines the
absolute ramification field $(K^r,v)$ of $(K,v)$. By \cite[(20.15) b)]{[En]},
$(L.K^r,v)$ is the absolute ramification field $(L^r,v)$ of $(L,v)$. By
Theorem~\ref{rdrram},
$(L^r,v)$ is an rdr field. From Lemma~\ref{ramfield} we know that $L.K^r|K^r$ is a 
finite tower of Galois extensions of degree $p$. By our assumption and  
Theorem~\ref{rdrram}, $(K^r,v)$ is not an rdr
field. Then there is a maximal field $(N,v)$ in the tower that is not an rdr field, 
and a Galois extension $(N',v)$ of $(N,v)$ of degree $p$ that is an rdr field. 

By part 1) of Lemma~\ref{basprop2}, $vN'$ contains $\frac{1}{p^\infty}\Z vp$. Since
$(N'|N,v)$ is a finite extension, also $vN$ contains $\frac{1}{p^\infty}\Z vp$. 
By part 1) of Lemma~\ref{basprop1}, $\erf(N',v)$ is perfect. As  $\erf(N',v)| 
\erf(N,v)$ is a finite extension, also $\erf(N,v)=Nv_0 v_p$ is perfect. Hence the same 
holds for $Nv$.

Since $(N,v)$ is not an rdr field, Proposition~\ref{dpd'1} shows that for every 
$d\in N$ with $vd \in \frac{1}{p^\infty}\Z vp$ and $0<vd\leq vp$ there must be some 
$b_d \in\cO_N^\times$ such that there is no $c\in N$ with $b_d-c^p\in d\cO_N\,$.
We choose $\eta_d\in \tilde{N}$ such that $\eta_d^p=b_d\,$. Then there is no $c\in N$ 
such that $v(\eta_d-c)\geq \frac{vd}{p}$ since this would imply $v(b_d-c^p)=
v(\eta_d^p-c^p)\geq vd$ as $\eta_d^p-c^p\equiv (\eta_d-c)^p \mod p\cO_N\,$. 
Lemma~\ref{maxnots} 
shows that $v(\eta_d-N)$ has no maximal element. Hence by Lemma~\ref{imm_deg_p},
$(N(\eta_d)|N,v)$ is a Galois defect extension, and by Proposition~\ref{distAS}, it 
has dependent defect. 

\pars
We distinguish two cases. First, let us assume that $(N'|N,v)$ is not a defect
extension. Then by Lemma~\ref{ivd}, $(N'(\eta_d)|N',v)$ is a Galois defect extension 
with $\dist(\eta_d,N')=\dist(\eta_d,N)$, 
which shows that also this extension has dependent defect. Therefore, $(N',v)$ is 
not an independent defect field and thus by 
Proposition~\ref{rdr->id}, it is not an rdr field. This contradicts our assumption.

\pars
Now let us assume that $(N'|N,v)$ is a defect extension. 
Since $K^r$ contains all $p$-th roots of unity, the same holds for $N$. Therefore, 
the extension $N'|N$ admits a Kummer generator $\eta$, and we can assume 
that it is a $1$-unit. Since $Nv_0 v_p$ is perfect, it follows that there is 
some $c\in N$ such that $v_0\circ v_p(\eta-c)>0$, and thus 
we can choose some $d\in N$ as above such that $\frac{vd}{p}\in v(\eta-N)$. It 
follows that 
\begin{equation}                          \label{smaller}
v(\eta_d-N)\subsetneq v(\eta-N)\>.
\end{equation}
This means that $\dist(\eta_d,N)<\dist(\eta,N)$. Note that $v(\sigma\eta_d-\eta_d)=
v(\sigma\eta-\eta)$ as both $\eta_d$ and $\eta$ are Kummer generators of value $0$ of the
extensions $N(\eta_d)|N$ and $N'|N$, respectively.

If $v(\eta_d-N')=v(\eta_d-N)$, 
then again by Lemma~\ref{ivd}, $(N'(\eta_d)|N',v)$ is a Galois defect extension with 
$\dist(\eta_d,N')=\dist(\eta_d,N)$, yielding a contradiction as before.

\pars
Suppose that $\eta_d\in N'$. Then inequality
(\ref{smaller}) leads to 
\begin{equation}                         \label{neq}
-v(\eta_d-N)+v(\sigma\eta_d-\eta_d) \>\neq\> -v(\eta-N)+v(\sigma\eta-\eta)\>,
\end{equation}
which in view of equation~(\ref{ram_gp}) together with Theorem~\ref{dist_galois_p}
is a contradiction. Hence we can assume that $\eta_d\notin N'$.

\pars
Now our proof will be complete once we show that $v(\eta_d-N')\ne v(\eta_d-N)$ is
impossible. In order to derive a contradiction, suppose that the two sets are not 
equal. Then there is some $\tilde\eta\in N'$ such that $v(\eta_d-\tilde\eta)\notin
v(\eta_d-N)$. Since $v(\eta_d-N)$ is an initial segment of $vN'=vN$, it follows 
that $v(\eta_d-\tilde\eta)>v(\eta_d-N)$. By part 1) of Lemma~\ref{relation},
\[
v(\eta_d-N)\>=\>v(\tilde\eta-N)
\] 
holds for all $\tilde\eta\in N'$ with $v(\eta_d-
\tilde\eta)>v(\eta_d-N)$. As $\tilde\eta\in N'\setminus N$ and $[N':N]=p$, also 
$\tilde\eta$ is a generator of $N'|N$. 

For $\sigma\in\Gal(\tilde K|K)$ with 
$\sigma\tilde\eta\ne\tilde\eta$, we compute:
\begin{equation}                    \label{sigma}
v(\sigma\tilde\eta-\tilde\eta)\>\geq\>\min\{v(\sigma\tilde\eta-\sigma\eta_d),
v(\sigma\eta_d-\eta_d),v(\eta_d-\tilde\eta)\}\>.
\end{equation}
As an algebraic extension of $(K^r,v)$, also $(N,v)$ 
is henselian. Hence we have that $v(\sigma\tilde\eta-\sigma\eta_d)=v
\sigma(\eta_d-\tilde\eta)=v(\eta_d-\tilde\eta)$. Suppose that
\[
v(\eta_d-\tilde\eta)\>\geq\>v(\sigma\eta_d-\eta_d)\>.
\]
As $(vN')_{vp}$ is $p$-divisible and $N'v$ is perfect, $v(\eta_d-N')$ does not have  
a maximum inside of $(vN')_{vp}\,$, so we may assume that $v(\eta_d-\tilde\eta)>
v(\sigma\eta_d-\eta_d)$. Thus in all cases, we may assume that $v(\eta_d-\tilde\eta)
\ne v(\sigma\eta_d-\eta_d)$. Hence by (\ref{sigma}),
\begin{equation}                              \label{vse-e}
v(\sigma\tilde\eta-\tilde\eta)\>=\>\min\{v(\sigma\eta_d-\eta_d),v(\eta_d-\tilde\eta)\}\>.
\end{equation}
If $v(\sigma\tilde\eta-\tilde\eta)=v(\sigma\eta_d-\eta_d)$, then 
$v(\sigma\tilde\eta-\tilde\eta)=v(\sigma\eta-\eta)$ and we obtain a contradiction exactly 
as in (\ref{neq}) with $\eta_d$ replaced by $\tilde\eta$. Hence we now assume that
\[
v(\sigma\tilde\eta-\tilde\eta)\>=\>v(\eta_d-\tilde\eta)
\><\>v(\sigma\eta_d-\eta_d)\>. 
\]
Again because $v(\eta_d-N')$ does not have a maximum inside of $(vN')_{vp}\,$, we can
choose $\tilde\eta_1\in N'$ such that 
\[
v(\eta_d-\tilde\eta_1)\> >\>v(\eta_d-\tilde\eta)\> >\>v(\eta_d-N)\>. 
\]
Like $\tilde\eta$, also $\tilde\eta_1$ is a generator of $N'|N$. With the same 
computations as before, we arrive
at (\ref{vse-e}) with $\tilde\eta$ replaced by $\tilde\eta_1\,$. We must have that 
$v(\eta_d-\tilde\eta_1)<v(\sigma\eta_d-\eta_d)$ since otherwise, we would obtain a
contradiction as before. Therefore,
\[
v(\sigma\tilde\eta_1-\tilde\eta_1)\>=\>v(\eta_d-\tilde\eta_1)
\]
and
\[
v(\tilde\eta_1-N)\>=\>v(\eta_d-N)\>=\>v(\tilde\eta-N)\>.
\]
Combining everything, we find:
\begin{eqnarray*}
-v(\tilde\eta_1-N)+v(\sigma\tilde\eta_1-\tilde\eta_1) &=&
-v(\eta_d-N)+v(\eta_d-\tilde\eta_1)\\
&\ne& -v(\eta_d-N)+v(\eta_d-\tilde\eta) \\
&=& -v(\tilde\eta-N)+v(\sigma\tilde\eta-\tilde\eta)\>,
\end{eqnarray*}
which again by equation~(\ref{ram_gp}) together with Theorem~\ref{dist_galois_p}
is a contradiction. 
\qed

%
%
\subsection{Proof of Theorem~\ref{KEindep} and of Proposition~\ref{idf}}         
\mbox{ }\sn
{\it Proof of Theorem~\ref{KEindep}:}
As before we set $K'=K(\zeta_p)$.
Then for any extension of $v$ to $\tilde{K}$, $(K(\zeta_p),v)$ is contained in 
the corresponding absolute ramification field. 
\sn
1): Assume that $(K,v)$ is an rdr field. The assertions on $(vK)_{vp}$ and 
$\erf(K,v)$ have been proven in Lemmas~\ref{basprop1} and~\ref{basprop2}.
By part 1) of Corollary~\ref{rdrramcor}, also 
$(K',v)$ is an rdr field. It follows from Proposition~\ref{rdr->id} that 
$(K',v)$ is an independent defect field. Thus by definition, $(K,v)$ is an 
independent defect field. The converse is the content of 
Proposition~\ref{chardr2}.

\mn
2): We note that every unibranched Galois extension of prime degree different from the
residue characteristic is automatically tame.

First, we assume that $(K,v)$ is a semitame field. Then by part 1) of 
Corollary~\ref{rdrramcor}, also $(K',v)$ is a semitame field, so $vK'$ is $p$-divisible. 
By Lemma~\ref{basprop1}, $K'v$ is perfect. Therefore, equation 
(\ref{feuniq}), with $K'$ in place of $K$, shows that every unibranched Galois 
extension $(L|K',v)$ of degree $p$ either has defect $p$, or satisfies
$[Lv:Kv]=p$ with $Lv|K'v$ a separable extension. In the latter case, the extension 
has no defect and is tame. Otherwise, it is a defect extension of
degree $p$. Then, as $(K',v)$ is an rdr field by Theorem~\ref{connprop}, part 1) of 
our theorem shows that it must have independent defect.

\pars
For the converse, we first show that our assumptions yield that $vK'$, and hence also 
$(vK')_{vp}$, is $p$-divisible, and that $\erf(K',v)$ is perfect. Indeed, if $\alpha\in
vK'$ is not divisible by $p$ and we take $a\in K'$ with $va=\alpha$, then taking a 
$p$-th root of $a$ induces a Galois extension that is neither tame nor immediate,
contradicting the hypothesis. The same holds if $a\in K'$ is such that $va=0$ and $av$ 
does not have a $p$-th root in $K'v$, hence $K'v$ is perfect.

Suppose that $\erf (K',v)$ is not perfect. Pick $a\in K'$ such that $av_0\circ v_p$ 
has no $p$-th root and choose some $b\in\tilde K$ such that $b^p=a$. Since $vK'$ is 
$p$-divisible, the same holds for $\ovl{v}(K'v_0\circ v_p)$. In addition, 
$(K'v_0\circ v_p)\ovl{v}=K'v$ is 
perfect. It follows that the extension $(K'v_0\circ v_p(bv_0\circ v_p)|K'v_0\circ v_p,
\ovl{v})$ is immediate of degree $p=[K'(b):K']$, which implies that also $(K'(b)|K',v)$
is immediate. Further, the
former extension is unibranched as it is purely inseparable. Since also the extension 
$(K'(b)|K,v_0\circ v_p)$ is unibranched as its inertia degree is $p$, also
$(K'(b)|K',v)$ is unibranched. By assumption, its defect must be independent
since defect extensions of degree $p$ are not tame. But then there must be $c\in K'$ such 
that $v(b-c)>\frac{vp}{p}$, whence $bv_0\circ v_p\in K'v_0\circ v_p$, contradiction.

Our assumption yields 
that every Galois defect extension of $(K',v)$ of degree $p$ is independent. 
Hence we obtain from part 1) that (DRvr) holds, so $(K',v)$ is a 
semitame field. By part 1) of Corollary~\ref{rdrramcor}, also $(K,v)$ is a semitame 
field.
\qed

\mn
{\it Proof of Proposition~\ref{idf}.}
Part 1) follows from Proposition~\ref{idfKK^r}. Part 2) has already been proved at the  
end of Section~\ref{sectASde}.
\qed

%
%
\subsection{Proof of Proposition~\ref{ax}}        
\mbox{ }\sn
It is well known that first order properties of the value group $vK$ of a valued field 
$(K,v)$ can be encoded in $(K,v)$ in the language of valued fields. The axiomatization
for (DRvp) and (DRst) is straightforward. Further, (DRvg) holds in an ordered abelian
group $(G,<)$ if and only if for each positive $\alpha\in G$ there is $\beta\in
G$ such that $2\beta\leq\alpha\leq 3\beta$.

If $(K,v)$ is of mixed characteristic, then (DRvr) is equivalent to the surjectivity 
of (\ref{OK/pOKeq}), and this in turn holds if and only if for each $a\in K$ with 
$va\geq 0$ there is $b\in K$ such that $v(a-b^p)\geq vp$. Hence
the classes of semitame, deeply ramified and rdr fields of mixed characteristic 
are first order axiomatizable.

If $(K,v)$ is of equal positive characteristic, then part 3) of Theorem~\ref{connprop} 
shows that semitame, deeply ramified and rdr fields form the same class. This class 
can be axiomatized by saying that $(K^p,v)$ is dense in $(K,v)$, or in other
words, for every $\alpha\in vK$ and every $a\in K$ there is $b\in K$ such that 
$v(a-b^p)\geq\alpha$.

In the case of equal characteristic 0, (DRvp), (DRvr) and (DRst) are trivial and all
valued fields are semitame and rdr fields, while the class of deeply ramified fields
consists of those which satisfy (DRvg).
\qed

\end{document}